\documentclass[a4paper, 11pt, oneside]{scrartcl}
\usepackage[utf8]{inputenc}
\usepackage[T1]{fontenc}

\usepackage{lmodern}
\usepackage{hyphsubst}%
\HyphSubstIfExists{ngerman-x-latest}{%
  \HyphSubstLet{ngerman}{ngerman-x-latest}}{}
\HyphSubstIfExists{german-x-latest}{%
  \HyphSubstLet{german}{german-x-latest}}{}
\usepackage[german, english]{babel} 
\usepackage{enumitem} 
\usepackage{microtype} 
\usepackage{color} 

\usepackage{listings}
\usepackage{amsmath}
\usepackage{amsfonts}
\usepackage{amssymb}
\usepackage{mathtools}
\usepackage{amsthm}
\usepackage{units} 
\usepackage{colonequals} 

\theoremstyle{plain}
\newtheorem{thm}{Theorem}
\newtheorem{lemma}[thm]{Lemma}
\newtheorem{corol}[thm]{Corollary}

\theoremstyle{definition}

\theoremstyle{remark}

\DeclareMathOperator*{\argmin}{arg\,min}
\DeclareMathOperator*{\argmax}{arg\,max}

\DeclareMathOperator*{\lexmin}{lex\,min}

\DeclareMathOperator*{\ext}{ext}
\DeclareMathOperator*{\interior}{int}
\DeclareMathOperator*{\conv}{conv}

\usepackage{tikz}
\usetikzlibrary{matrix,arrows,calc,intersections, positioning, trees, shapes, backgrounds, decorations.pathmorphing, cd, arrows.meta}

\usepackage{layout}

\usepackage{algorithmic}
\usepackage[section]{algorithm}

\usepackage{hyperref}
\hypersetup{
	pdfstartview=Fit,	
	bookmarksopen=true,
	bookmarksnumbered=true,
	pdftoolbar=true,
	pdfmenubar=true,
	pdfauthor={Fabian Chlumsky-Harttmann},
	pdftitle={Approaches for Biobjective Integer Linear Robust Optimization},
	breaklinks=true
}

\usepackage[normalem]{ulem}

\renewcommand{\phi}{\varphi}
\renewcommand{\theta}{\vartheta}
\renewcommand{\epsilon}{\varepsilon}
\newcommand{\R}{\mathbb{R}}
\newcommand{\Z}{\mathbb{Z}}

\newcommand{\cU}{\mathcal{U}}
\newcommand{\cX}{\mathcal{X}}
\newcommand{\cY}{\mathcal{Y}}

\newcommand{\Rep}{\mathcal{R}}
\newcommand{\PsingleU}[1]{\textup{P}^\textup{single}(#1)}

\usepackage[utf8]{inputenc}
\usepackage[normalem]{ulem} 

\definecolor{DklCyan}{rgb}{0.12,0.47,0.58}

\usepackage[official]{eurosym}
\usepackage{centernot}
\usepackage[weather]{ifsym}

\newcommand{\hide}[1]{}

\definecolor{myred}{RGB}{255,0,0}
\definecolor{mygreen}{RGB}{0,255,0}
\definecolor{myblue}{RGB}{0,0,255}
\definecolor{myyellow}{RGB}{255,255,0}
\definecolor{mygreenblue}{RGB}{0,255,127}
\definecolor{mypurple}{RGB}{85,26,139}
\definecolor{mygrey}{RGB}{100,100,100}

\usetikzlibrary{arrows.meta}
\tikzset{%
  >={Latex[width=2mm,length=2mm]},
            base/.style = {rectangle, rounded corners, draw=black,
                           minimum width=4cm, minimum height=1cm,
                           text centered, font=\sffamily},
  activityStarts/.style = {base, fill=blue!30},
       startstop/.style = {base, fill=red!30},
    activityRuns/.style = {base, fill=green!30},
         process/.style = {base, minimum width=2.5cm, fill=orange!15,
                           font=\ttfamily},
}

\usepackage[backend=bibtex8, style=alphabetic, maxbibnames=100, firstinits=true]{biblatex}
\usepackage[margin=0.8in]{geometry}
\usepackage{twoopt}
\usepackage{tikz}
\usetikzlibrary{fit}

\newcommand\tikzmark[1]{%
  \tikz[remember picture,overlay]\node[inner xsep=0pt] (#1) {};}
\newcommandtwoopt\Textboxred[5][0cm][8cm]{%
\begin{tikzpicture}[remember picture,overlay]
  \coordinate (aux) at ([xshift=#1]#4);
  \node[inner ysep=3pt,yshift=0.6ex,draw=red,thick,
    fit=(#3) (aux),baseline] 
    (box) {};
  \node[text width=#2,anchor=north east,
    font=\sffamily\footnotesize,align=right] 
    at (box.north east) {#5};
\end{tikzpicture}%
}
\newcommandtwoopt\Textboxblue[5][0cm][8cm]{%
\begin{tikzpicture}[remember picture,overlay]
  \coordinate (aux) at ([xshift=#1]#4);
  \node[inner ysep=3pt,yshift=0.6ex,draw=blue,thick,
    fit=(#3) (aux),baseline] 
    (box) {};
  \node[text width=#2,anchor=north east,
    font=\sffamily\footnotesize,align=right] 
    at (box.north east) {#5};
\end{tikzpicture}%
}

\makeatletter
\newcommand{\leqnomode}{\tagsleft@true\let\veqno\@@leqno}
\newcommand{\reqnomode}{\tagsleft@false\let\veqno\@@eqno}
\makeatother

\usepackage{authblk}

\author[1]{Fabian Chlumsky-Harttmann\footnote{supported by DFG under grant SCHO 1140/11-1}}
\author[2]{Marie Schmidt}
\author[1,3]{Anita Schöbel}
\affil[1]{Faculty of Mathematics, University of Kaiserslautern-Landau}
\affil[2]{Faculty of Mathematics and Computer Science, University of Würzburg}
\affil[3]{Fraunhofer Institute for Industrial Mathematics ITWM}

\title{Approaches for Biobjective Integer Linear Robust Optimization}
\date{}

\begin{document}
\maketitle

\paragraph{Abstract}
Real-world optimization problems often do not just involve multiple objectives but also
uncertain parameters.
In this case, the goal is to find Pareto-optimal solutions that are
robust, i.e., reasonably good under all possible realizations of the uncertain data.
Such solutions have been studied in many papers within the last ten years and are called \emph{robust efficient}.
However, solution methods for finding robust efficient solutions are scarce.
In this paper, we develop three algorithms for determining robust efficient solutions to
biobjective mixed-integer linear robust optimization problems. 

To this end, we draw from methods for both multiobjective optimization and robust optimization:
dichotomic search for biobjective mixed-integer optimization problems and an
optimization-pessimization approach from (single-objective) robust optimization, which
iteratively adds scenarios and thereby increases the uncertainty set. 
We propose two algorithms that combine dichotomic search with the optimization-pessimization
method as well as a dichotomic search method for biobjective linear robust optimization that
exploits duality.
On the way we derive some other results:
We extend dichotomic search from biobjective linear problems to biobjective linear
minmax problems and generalize the optimization-pessimization method 
from single-objective to multi-objective robust optimization problems.

We implemented and tested the three algorithms on linear and integer linear
instances and discuss their respective strengths and weaknesses.

\section{Introduction}

Real-world optimization problems are often complicated by two issues:
First, in many cases decision makers have not only one but multiple objectives. Second, the optimization problems may involve uncertainty -- be it through prediction errors about parameters like demand, that will only be known in the future, or measurement errors.
These two issues are treated in the fields of multiobjective optimization and robust optimization.
\medskip

In order to do deal with problems that are both uncertain and multiobjective, 
multiobjective robust optimization has been studied for more than ten years leading to various models and theoretical results.
However, research into methods of actually solving such problems is still in its initial stages.
In this paper we propose three algorithms for computing robust
efficient solutions for uncertain biobjective mixed-integer linear optimization problems.
\medskip

In order to find a good solution for an uncertain multiobjective problem, a notion of what constitutes a \emph{robust efficient} solution has to be formulated first. 
This is not trivial since there is no straightforward way to generalize the concept of Pareto optimality used in multiobjective optimization to uncertain multiobjective problems or to generalize the notion of robustness to multiobjective problems. 
Over the years, several concepts for robust multiobjective efficiency have been proposed, (see \cite{IdeSchoe13, Wieceksurvey} for surveys).
The oldest among them 
is the notion of \emph{flimsily efficient} (sometimes: \emph{possibly efficient}) and \emph{highly efficient} (sometimes: \emph{necessarily efficient}) solutions (see, e.g., \cite{Bitran80, Inuiguchi96, robipa, EngauSigler2019}) describing solutions that are efficient for at least one or for all considered scenarios, respectively.
Other notable concepts include
\emph{regret-robust efficiency}
(see \cite{Rivaz2013,XidMavHasZop2017regretMORO,GroetznerWerner22}),
\emph{multi-scenario efficiency} (see \cite{BotteSchoebel2019}),
\emph{lightly robust efficiency} (see \cite{robipa,IdeSchoe13}),
\emph{local efficiency w.r.t. the robust counterpart} (see \cite{Chuong16}), and
three different generalizations of minmax robustness to multiobjective problems
called \emph{set-based} (see \cite{EIS13}), \emph{hull-based} (see \cite{BokFred17}) and
\emph{point-based minmax efficiency}.
The latter concept has been introduced by Kuroiwa and Lee (see \cite{kurlee12}) and is used in this paper.

For point-based minmax robust efficiency, many theoretical results exist: 
Goberna, Jeyakumar, Li and Vicente-Pérez consider specific forms of data uncertainty (box data uncertainty, norm data uncertainty, ellipsoidal uncertainty) and provide deterministic reformulations (see \cite{GJLV15}). Box uncertainty with a limited sum of deviations has been considered in
\cite{Hassanzadeh13}.
In \cite{antczak2020approximate} 
necessary and sufficient conditions for robust $\epsilon$-efficient solutions for uncertain nonsmooth
multiobjective optimization problems are established, but no algorithmic method is provided.
In \cite{wei2019characterizations,wei2020unified} separation results and some characterizations of optimality are developed, and the robustness gap for point-based minmax robust efficiency has been
introduced in \cite{kruger2023point}. 
The \emph{price of robustness} has been defined in \cite{SchZhue2021}.
Point-based minmax robust efficiency has been generalized to \emph{efficiency w.r.t. to a general cone} (see \cite{Wang2015RobustCA,IKKST13}) and it has been applied to decision robustness in \cite{RegRob}. 

As general algorithmic idea, many authors suggest scalarization approaches transferring a robust multiobjective problem to a single-objective robust problem, e.g, \cite{EIS13, IKKST13, GJLV15}, but the approaches proposed in those papers
are still on an abstract level and only capable of finding some robust efficient solutions 
while in this paper we give concrete algorithms for determining a representative set of all supported
robust efficient solutions.
Other algorithmic approaches consider special cases, e.g.,
cardinality-constrained uncertainty for combinatorial problems (see \cite{RSST2018}), 
uncertain multiobjective shortest-path problems \cite{RSST2018-SPLabeling} or 
cardinality-constrained box uncertainty in the context of portfolio selection problems \cite{hassanzadeh2014robust}.

\medskip

The remainder of the paper is organized as follows. In Section~\ref{sec:preliminaries} we derive a biobjective integer linear minmax optimization problem as robust counterpart and collect other necessary preliminaries.
Section~\ref{sec:DichotomicSearch-intro} considers the problem first and foremost as a biobjective problem. The well-known dichotomic search algorithm for biobjective problems is briefly summarized before we show how it can be extended to robust problems.
The opposite approach is taken in Section~\ref{sec:OptPess-intro}, where the problem is considered from a robust optimization perspective. An optimization-pessimization approach for (single-objective) robust optimization is reviewed and then extended to multiobjective problems.

In Section~\ref{sec:Algos} we combine dichotomic search and optimization-pessimization
and receive two different methods for finding robust efficient solutions.
For the special case of a bilinear continuous objective function, we additionally develop
a dual approach together with dichotomic search in Section~\ref{sec:dualize}.
Numerical results are given in Section~\ref{sec:NumericalResults} and, finally, some conclusions are drawn and suggestions for further research are formulated in Section~\ref{sec:Conclusion}.

\section{Problem definition and preliminaries } \label{sec:preliminaries}

In this section we briefly review multiobjective robust optimization. 
We start with restating some definitions from robust optimization and multiobjective optimization which we then combine to
the emerging topic of multiobjective robust optimization. We define what a
robust efficient solution to an uncertain multi-objective problem is
and from this we derive the biobjective mixed-integer linear robust optimization problem \eqref{eq:GrundproblemBRO} --- the problem 
to be solved in this paper.
We finally recall some concepts of multiobjective optimization which are needed
later.

\paragraph{Single-objective robust optimization.}
Robust optimization deals with uncertain optimization problems, i.e., problems with
some uncertain parameters $\xi \in \R^m$ which depend on measurements, future developments, delays
or other uncertainties.
Every $\xi$ is called a \emph{scenario}.
As usual in robust optimization, we assume that the set $\cU \subseteq \R^m$ of all possible scenarios is known.
We call $\cU$ \emph{uncertainty set}.
A single-objective uncertain problem
is described by a family of parameterized optimization problems
\begin{equation} \label{eq:UncertainProblem}
\left\{
\min_{x \in \cX} h(x,\xi)
\right\}_{\xi \in \cU}
\end{equation}
with $\cX \subseteq \R^n$ and $h \colon \cX \times \cU \to \R$.

There is usually no solution that is optimal for all scenarios.
Instead one aims to find \emph{robust solutions} which are reasonably good for all (or most) scenarios.
Out of many robustness concepts that have been defined 
(see, e.g., \cite{GoeSchoe13-AE} for an overview on different robustness concepts),
minmax robustness is one of the most commonly used.
For a detailed account of the subject, we refer to \cite{RObook}.
A solution to problem \eqref{eq:UncertainProblem} is called \emph{(minmax) robust optimal}
if it is an optimal solution to its \emph{robust counterpart}
\begin{align} \label{eq:MinSupProblem}
\PsingleU{\cU}
&&
\min_{x \in \cX} \sup_{\xi \in \cU} h(x,\xi).
&&
\end{align}

\paragraph{Multiobjective (deterministic) problems.} Now let us turn to multiobjective problems
\begin{equation} \label{eq:MultiobjectiveProblem}
\min_{x \in \cX}
\underbrace{\begin{pmatrix}
g_1(x) \\ g_2(x) \\ \vdots \\ g_p(x) 
\end{pmatrix}}_{\eqqcolon g(x)}
\end{equation}
with $\cX \subseteq \R^n$ and $g_i \colon \cX \to \R$, $i=1,2,\dots,p$.

Instead of minimizing a scalar objective function, we have to compare vectors in
order to find an optimal solution. As common in multi-objective optimization, we use
the following vector relations: For two vectors $y,\bar{y} \in \R^p$,
the ordering relations $<$ and $\leq$ are meant to be component-wise. Furthermore,
if $y \leq \bar{y}$ and $y \not = \bar{y}$, we write $y \preceq \bar{y}$ and say that $y$
\emph{dominates} $\bar{y}$.
Accordingly, we define $\R_\succeq^p  \coloneqq \{r \in \R^p \colon r \succeq 0\}$ and $\R^p_\geq, \R^p_>$. 
Biobjective optimization is the special case of multiobjective optimization with $p=2$.

The most important concept for multiobjective optimization is \emph{efficiency} (also
called \emph{Pareto optimality}).
Given a multiobjective problem
\eqref{eq:MultiobjectiveProblem}
a solution $x \in \cX $ is called \emph{efficient} and its image $g(x) \in \cY \coloneqq g(\cX) \subset \R^p$ is called \emph{nondominated} if no solution $x^\prime \in \cX \setminus \{x\}$ exists, such that $g(x^\prime)$ dominates $g(x)$.
By $\cY_\textup{N}$ we denote the set of nondominated points.
These points form the \emph{Pareto frontier}.

\medskip

\paragraph{Multiobjective robust optimization.}
Real-world optimization problems often have multiple objective functions \emph{and} uncertain parameters.
We consider multiobjective uncertain optimization problems which depend on a scenario
$\xi \in \cU \subseteq \R^m$ 
\begin{align} \label{eq:ParametrizedMultiobjectiveProblem}
\textup{P}(\xi)
&& \min_{x \in \cX}
\underbrace{%
\begin{pmatrix}
f_1(x,\xi) \\ f_2(x,\xi) \\ \vdots \\ f_p(x,\xi)
\end{pmatrix}
}_{\eqqcolon f(x,\xi)}
&&
\end{align}
with $\cX \subseteq \R^n$, $f_i \colon \cX \times \cU \to \R$, $i=1,2,\dots,p$.
Analogously to single-objective optimization, one calls the parameterized family
\begin{equation} \label{eq:UncertainMultiobjectiveProblem} 
\{ P(\xi) \colon \xi \in \cU \}
\end{equation}  
an \emph{uncertain multiobjective optimization problem}. 
We are interested in finding efficient solutions to the uncertain multiobjective optimization
problem, which are robust. 

Several ways to generalize minmax robustness to multiobjective
uncertain problems have been proposed (see \cite{IdeSchoe13,Wieceksurvey}
for surveys). In this article we aim to find \emph{point-based minmax robust efficient}
solutions as defined in \cite{kurlee12,robipa}.
A solution is called \emph{point-based minmax robust efficient}
(from now on: \emph{robust efficient}) if it is an efficient solution to
\begin{align} \label{eq:Grundproblem-p-kriteriell}
P(\cU)
&&
\min_{x \in \cX}
   \begin{pmatrix}
\sup_{\xi \in \cU}f_1(x,\xi) \\ \sup_{\xi \in \cU} f_2(x,\xi) \\ \vdots \\ \sup_{\xi \in \cU} f_p(x,\xi) \\ 
\end{pmatrix}
.
&&
\end{align}
$P(\cU)$ is called the \emph{robust counterpart} of the uncertain multiobjective optimization problem~\eqref{eq:UncertainMultiobjectiveProblem} or just \emph{multiobjective robust} problem. Since $\cU$ is varied within some of the proposed algorithms, we
refer to the specific set $\cU$ in the notation $P(\cU)$.

To improve readability, we define 
$f_i^\cU(x) \coloneqq \sup_{\xi \in \cU} f_i(x,\xi)$, $i=1,2,\dots, p$
and set 
$f^\cU(x) \eqqcolon (f_1^\cU(x), f_2^\cU(x), \dots, f_p^\cU(x))^t$ as the vector containing 
the $p$ objective functions.
Problem \eqref{eq:Grundproblem-p-kriteriell} can hence be interpreted as a (deterministic) multiobjective problem of type~\eqref{eq:MultiobjectiveProblem} with $g \coloneqq f^\cU$ as objective function.
This point of view is used in Section~\ref{sec:DichotomicSearch} and in Algorithms~\ref{algo:MOA} and \ref{algo:dichotomic_with_dual}.

Finally, note that
\begin{equation} \label{eq:U'relaxation}
f^{\cU^\prime}(x) \leq f^\cU(x) \textup{ for } \cU^\prime \subseteq \cU \textup.
\end{equation}

\paragraph{The problem to be solved: (BRO) -- biobjective mixed-integer linear robust optimization.}
We consider uncertain biobjective optimization problems with $p=2$. 
Their robust counterpart is given in \eqref{eq:Grundproblem-p-kriteriell}, i.e.,
we receive the following biobjective mixed-integer linear program with minmax objective function,

\begin{align}
\textup{BRO}(\cU)
&&
\min_{x \in \cX}
  \begin{pmatrix}
   \sup_{\xi \in \cU} f_1(x,\xi)\\
   \sup_{\xi \in \cU} f_2(x,\xi)
 \end{pmatrix}
\textup{.}
&&
\tag{BRO}\label{eq:GrundproblemBRO}
\end{align}

Our goal is to determine the Pareto frontier and the associated efficient solutions of $\textup{BRO}(\cU)$.

For $\textup{BRO}(\cU)$ we always assume the following:
\begin{itemize}
\item (BRO-1) a feasible set $\cX=P \cap (\Z^{k} \times \R^{n-k})$ where $P \subseteq \R^n$ is a polytope and $0 \leq k \leq n$,
\item (BRO-2) a polytope or finite set $\cU \subseteq \R^m$, and
\item (BRO-3) functions $f_1, f_2 \colon \cX \times\cU\to \R$
    which are linear in $x$ for every fixed $\xi \in \cU$ and quasi-convex and continuous in $\xi$ for every fixed $x \in \cX$.
\end{itemize}

Under the latter two assumptions, (BRO-2) and (BRO-3), the supremum in the definition of \eqref{eq:GrundproblemBRO} is always attained and we can write maximum instead, i.e., 
$f^\cU_i(x)=\max_{\xi \in \cU} f_i(x,\xi)$ for $x\in\cX$, $i=1,2$.
(BRO-3) guarantees that $f_i \colon \cX \times \cU \to \R$ is jointly continuous in $(x,\xi)$ (see, e.g., \cite{kruse1969joint}). 
Finally, the feasible set $\cX$ determines the type of the problem at hand:
For $k=0$ the problem is a (pure) \emph{linear} minmax problem,
for $k=n$ the problem is an \emph{integer linear} minmax problem and 
for $1 \leq k <n$ we have a \emph{mixed-integer linear} minmax problem.

\paragraph{Concepts from multi-objective optimization.}
We recall some concepts from multiobjective optimization which we need in this paper.
Consider the deterministic multiobjective problem \eqref{eq:MultiobjectiveProblem}.
We first define two special types of efficient solutions, namely \emph{supported efficient} and \emph{extreme supported efficient} solutions.
There exist slightly different characterizations of these solutions.
We use the definitions of Özpeynirci and Köksalan (see \cite{ozpeynirci2010exact}) and call a point $y\in \cY$ 
\emph{extreme supported nondominated}, if there is no convex combination of
nondominated points $y^{(1)},y^{(2)},\dots,y^{(n)} \in \cY \setminus \{y\}$ such that
$\sum_{i=1}^n \lambda_i y^{(i)} \leq y$.
We call a point \emph{supported nondominated}, if there is no convex combination of nondominated points $y^{(1)},y^{(2)},\dots,y^{(n)} \in \cY \setminus \{y\}$ such that $\sum_{i=1}^n \lambda_i y^{(i)} < y$.
A solution $x \in \cX$ is called \emph{(extreme) supported efficient}, if $y=g(x)$ is (extreme) supported nondominated.
A supported efficient solution can be found by solving the scalarized weighted-sum problem
\begin{align*}
P(\lambda)
&&
\min_{x \in \cX} \lambda_1 g_1(x) + \lambda_2 g_2(x) + \dots + \lambda_p g_p(x)
&&
\end{align*}
for some weight vector $\lambda\in \R^p_{\succeq 0}$.
We use $\cY_{\textup{ESN}}$ to denote the set of extreme supported nondominated points.
Any subset of $\cX$ whose image under $g$ is the set of nondominated points $\cY_{\textup{N}}$ is called a \emph{representative set}; a set whose image under $g$ equals the set of extreme supported nondominated points $\cY_{\textup{ESN}}$ is called  \emph{representative set for the extreme supported efficient solutions}.
\medskip

In the following we state two properties that are essential to prove some of our results.
The first is the existence of the \emph{ideal point}
\begin{equation*}
y^\text{I} \coloneqq
\begin{pmatrix}
\min_{x \in \cX} g_1(x) \\
\min_{x \in \cX} g_2(x) \\
\vdots \\
\min_{x \in \cX} g_p(x) 
\end{pmatrix}
\end{equation*}
for \eqref{eq:MultiobjectiveProblem}.
We say that the \emph{ideal point property} is
satisfied if an ideal point exists, i.e.,
\leqnomode
\begin{equation} \label{eq:ideal-point-existent} \tag{ideal}
\min_{x \in \cX} g_i(x) \textup{ exists for }i=1,2,\dots,p.
\end{equation}
\reqnomode

The second property we need is
the \emph{domination property} (see \cite{Henig1986}).
\leqnomode
\begin{equation} \label{eq:domination-property} \tag{dom}
\text{For all $y\in\cY \setminus \cY_\textup{N}$, there exists a point $y^\prime \in \cY_\textup{N}$ with $y^\prime \preceq y$.}
\end{equation}
\reqnomode

The following result is well known.
\begin{lemma}
\label{lem:DomProperty-IdealPoint}  
Let a multiobjective problem \eqref{eq:MultiobjectiveProblem} be given.
If $\cX$ is finite, or if $\cX$ is compact and $g$ is continuous, then
both, \eqref{eq:ideal-point-existent} and \eqref{eq:domination-property} hold.
\end{lemma}
\begin{proof}
For \eqref{eq:ideal-point-existent} this is due to Weierstrass' Extreme Value Theorem,  for \eqref{eq:domination-property} we refer to \cite{Henig1986}.
\end{proof}


\paragraph{Domination and ideal point property for multiobjective robust optimization problems.}
We conclude this section by discussing under which assumptions 
\eqref{eq:ideal-point-existent} and \eqref{eq:domination-property} 
are satisfied for \emph{robust} multiobjective problems (see \eqref{eq:Grundproblem-p-kriteriell}), 
i.e., for the case that the objective functions of \eqref{eq:MultiobjectiveProblem} are given as $g=f^\cU$.
For a discussion of \eqref{eq:domination-property} in the context of multiobjective robust optimization, see also \cite{SchZhue2021}.

\begin{thm}\label{thm:DomProperty-IdealPoint}
Let either
\begin{enumerate}[label=(\roman*)]
\item $\cX$ and $\cU$ be compact and $f$ jointly continuous in $\cX$ and $\cU$,
\item $\cX$ be finite, $\cU$ compact and $f$ continuous in $\cU$ for every fixed $x \in \cX$,
\item $\cU$ be finite, $\cX$ compact and $f$ continuous in $\cX$ for every fixed $\xi \in \cU$, or
\item $\cX$ and $\cU$ both be finite.
\end{enumerate}
Then both, \eqref{eq:domination-property} and \eqref{eq:ideal-point-existent}
are satisfied for a
multiobjective robust optimization problem \eqref{eq:MultiobjectiveProblem}.
\end{thm}

\begin{proof}
  We set $g_i(x)\coloneqq\sup_{\xi \in \cU} f_i(x,\xi)$, $i=1,2,\dots,p$, and distinguish two cases: 
\begin{itemize}
\item[(a)]
$\cX$ is finite:
Due to Lemma~\ref{lem:DomProperty-IdealPoint}, \eqref{eq:domination-property} and \eqref{eq:ideal-point-existent} hold if $g_i(x)=\sup_{\xi \in \cU} f_i(x,\xi)$ exists for all $x \in \cX$.
This is the case since either $\cU$ is finite or $\cU$ is compact and $f_i(x,\cdot)$ continuous for every fixed $x \in \cX$.
\item[(b)]
$\cX$ is compact:
In this case, Lemma~\ref{lem:DomProperty-IdealPoint} requires that $g_i(x)$ is continuous.
This holds since
\begin{itemize}
\item either $\cU$ is finite, hence $g(x)$ is continuous as the maximum of a finite set of continuous functions $f(\cdot,\xi)$, $\xi \in \cU$,
\item or $\cU$ is compact and $f$ is jointly continuous in $(x,\xi)$ and hence again, $g_i(x)$ is continuous.
\end{itemize}
\end{itemize} 
\end{proof}

We conclude that \eqref{eq:domination-property} and \eqref{eq:ideal-point-existent} hold for \eqref{eq:GrundproblemBRO}.

\begin{corol} \label{cor:PropertiesForBRO}
\eqref{eq:GrundproblemBRO} satisfies both, \eqref{eq:domination-property} and \eqref{eq:ideal-point-existent}.
\end{corol}

\begin{proof}
By the assumptions of \eqref{eq:GrundproblemBRO}, $\cX$ and $\cU$ are both compact (see (BRO-1) and (BRO-2)) and $f_i:\cX \times \cU \to \R$ is jointly continuous in $(x,\xi)$ for $i=1,2$ (BRO-3). Theorem~\ref{thm:DomProperty-IdealPoint} hence gives the result.
\end{proof}

\section{Dichotomic search for biobjective minmax optimization} \label{sec:DichotomicSearch}

In this section we view our problem \eqref{eq:GrundproblemBRO}
as a deterministic biobjective (mixed-integer) linear minmax problem.
First, in Section~\ref{sec:DichotomicSearch-intro},
we repeat dichotomic search from literature.
In Section~\ref{sec:DichotomicSearch-minmax} we generalize this method 
from biobjective mixed-integer linear optimization
to biobjective mixed-integer linear \emph{minmax} optimization,
i.e., to problems of type \eqref{eq:GrundproblemBRO}.

\subsection{Dichotomic search for biobjective mixed-integer linear optimization} \label{sec:DichotomicSearch-intro}

We consider a special case of \eqref{eq:MultiobjectiveProblem}, namely biobjective linear mixed-integer optimization problems,
\begin{equation} \label{eq:bioobjective-linear-mip}
\min_{x \in \cX}
\begin{pmatrix} g_1(x) \\ g_2(x) \end{pmatrix}
\text.
\end{equation}

The feasible set $\cX \subseteq \R^n$ is a polyhedron and as in \eqref{eq:GrundproblemBRO} it is intersected with $\Z^k \times \R^{n-k}$.
The objective functions $g_1,g_2 \colon\cX \to \R$ are linear functions.
\medskip

A well-known approach to solve such problems is \emph{dichotomic search},
formulated in Algorithm~\ref{algo:DichotomicSearch}.
The method has first been published by Aneja and Nair in 1979 (see \cite{AnejaNair79}) and Cohon (see \cite{cohon1978multiobjective}) for more specific problem classes and is now part of multi-objective folklore and sometimes also known as \emph{Aneja and Nair's bicriteria method} (e.g., \cite{UlunguSurvey94}) or \emph{CAN method} (e.g., \cite{ozpeynirci2010exact}).
Most frequently, it is used to solve biobjective \emph{linear} problems.
However, it can also be applied to
biobjective \emph{mixed-integer linear} problems where it 
determines all
extreme supported efficient nondominated points $Y^\ast$ and a representative set of extreme supported nondominated solutions $X^\ast$.
Dichotomic search takes advantage of the fact that in $\R^2$ sorting nondominated solutions with respect to their first coordinates is the same as reverse sorting by the second coordinate, i.e., for two nondominated solutions $y^l, y^r \in \cY \subset \R^2$, $y^l_1 < y^r_1$ implies $y^r_2 > y^l_2$. 
The idea is to start with the lexicographically optimal solutions and then in each step find a supported non-dominated point ``between'' two given supported non-dominated points. The method proceeds iteratively until all extreme supported nondominated points are identified.
Algorithmically, first, the lexicographic optimal solutions $x^L,x^R$ for \eqref{eq:bioobjective-linear-mip} are computed.
After that, in each iteration, a tuple $(y^l, y^r)$ of two points known to be supported nondominated is taken and $\lambda=(y^l_2-y^r_2,y^r_1-y^l_1)$, corresponding to the slope $\frac{y^l_2-y^r_2}{y^r_1-y^l_1}$ of the line segment from $y^l$ to $y^r$, is chosen. 
Solving the corresponding weighted-sum (scalarized) problem
\[
\min_{x \in \cX} \lambda^t g(x)
\]
either finds a new supported nondominated point between $y^l$ and $y^r$ or certifies that there is no such point. 
The algorithm terminates when all extreme supported nondominated points
-- each with a corresponding extreme supported efficient solution --
have been discovered.
It might find also supported nondominated points which are not extreme supported nondominated,
but these can be easily identified and removed.
  

\begin{algorithm}[hbt!]
\small
	\begin{algorithmic}[1] 
	\REQUIRE{Biobjective mixed-integer linear optimization problem \eqref{eq:bioobjective-linear-mip}.}
	\ENSURE{Feasible set $\cX$ is a polyhedron intersected with $\R^{n-k} \times \Z^k$ for some $k \in \{0,\ldots,n\}$.}
	\STATE{Initialize $\mathcal{L} \coloneqq \emptyset$.}     \COMMENT{$\mathcal{L}$ will contain list of tuple images $(y^l,y^r)$ satisfying $y^l_1 < y^r_1, y^l_2 > y^r_2$}

	\tikzmark{ds-lexicographic-start}
	
	\STATE{Compute $\epsilon_1 \coloneqq  \min_{x \in \cX} g_1(x)$.}
	\STATE{Determine $x^L \in \argmin_{x \in \cX} \{ g_2(x) \colon g_1(x) \leq \epsilon_1 \}$.}
	\STATE{Set $y^L \coloneqq g(x^L) $.}  
	\STATE{Compute $\epsilon_2 \coloneqq  \min_{x \in \cX} g_2(x)$.}
	\STATE{Determine $x^R \in \argmin_{x \in \cX} \{ g_1(x) \colon g_2(x) \leq \epsilon_2 \}$.}
	\STATE{Set $y^R \coloneqq g(x^R)$.}

	\hfill~\tikzmark{ds-lexicographic-end}

	\IF{$y^L=y^R$}
		\STATE{STOP. Only one nondominated image found.}
		\RETURN{$Y^\ast = \{y^L\}, X^\ast = \{x^L\}$.}
	\ELSE
		\STATE{$Y^\ast = \{y^L,y^R\}, X^\ast = \{x^L,x^R\}, \mathcal{L}=\{(y^L,y^R)\}$.}
	\ENDIF
	\WHILE{$L \not = \emptyset $}
		\STATE{Remove element $(y^l,y^r)$ from $\mathcal{L}$.}
		\STATE{Compute $\lambda \coloneqq (y^l_2-y^r_2,y^r_1-y^l_1)$.}
		
		\tikzmark{ds-lambda-start}
		
		\STATE{Determine $x^\ast \in \argmin_{x\in\cX} \lambda^T g(x)$.}
		\STATE{Set $y^\ast \coloneqq g(x^\ast)$.}
		
		\hfill~\tikzmark{ds-lambda-end}
		
		\IF{$\lambda^T y^\ast < \lambda^T y^l$.}
			\STATE{Add $y^\ast$ to $Y^\ast$, add $x^\ast$ to $X^\ast$.}
			\STATE{Add $(y^l,y^\ast),(y^\ast,y^r)$ to $\mathcal{L}$.}
		\ENDIF
	\ENDWHILE
	\RETURN{$Y^\ast$: contains all extreme supported nondominated points.}
	\RETURN{$X^\ast$: contains a representative set of extreme supported efficient solutions.}
\end{algorithmic}
\Textboxred{ds-lexicographic-start}{ds-lexicographic-end}{Determine lexicographic solutions}
\Textboxred{ds-lambda-start}{ds-lambda-end}{Solve weighted-sum problem for weights $\lambda$}
\caption{Dichotomic search} \label{algo:DichotomicSearch}
\end{algorithm}


Finiteness and correctness of Algorithm~\ref{algo:DichotomicSearch} follow from the considerations above which are derived from the literature (e.g., \cite{przybylski2019simple,ozpeynirci2010exact}) 
and are stated in the following lemma.
The lemma is valid if \eqref{eq:bioobjective-linear-mip} satisfies \eqref{eq:ideal-point-existent}.
This is a slight generalization to \cite{ozpeynirci2010exact} who assumed that \eqref{eq:bioobjective-linear-mip} is bounded by the origin, i.e., $g_i(x)$, $i=1,2$, are non-negative for all $x\in \cX$.

\begin{lemma}[e.g., \cite{ozpeynirci2010exact}] \label{lem:Dichotomic}
Let a biobjective problem as in \eqref{eq:bioobjective-linear-mip} be given, i.e., 
\begin{itemize}
\item with linear objectives $g_1$, $g_2$ and
\item a feasible set $\cX$ that is a polyhedron intersected with $\R^{n-k} \times \Z^k$.
\item Furthermore, let \eqref{eq:ideal-point-existent} hold for
  \eqref{eq:bioobjective-linear-mip}.
\end{itemize}  
Then Algorithm~\ref{algo:DichotomicSearch} returns a set $Y^\ast$ containing all extreme supported nondominated
points and a set $X^\ast$ containing a representative set of extreme supported efficient solutions
after $2|Y^\ast| - 3$ iterations (lines 15--22) if $|Y^\ast| > 2$ and zero iterations if $|Y^\ast|=1$.
\end{lemma}

It is known that
in the case of biobjective linear optimization problems, 
the set of all extreme supported nondominated points and
a representative set of extreme supported efficient solutions
can be used to construct all nondominated points
and a representative set of efficient solutions,
respectively.
We will show a related result for \eqref{eq:GrundproblemBRO} in Lemma~\ref{lem:Construct_Paretofront_from_ESN-Points} in Section~\ref{sec:Algos}.

\subsection{Dichotomic search for biobjective mixed-integer linear \emph{minmax} optimization} \label{sec:DichotomicSearch-minmax}

Our goal is to apply dichotomic search to \eqref{eq:GrundproblemBRO}, i.e., to a biobjective
mixed-integer linear robust optimization problem which is given as the minmax problem introduced
in Section~\ref{sec:preliminaries}

\begin{align*} \tag{\ref{eq:GrundproblemBRO} revisited}
\textup{BRO}(\cU)
&&  
\min_{x \in \cX}
  \begin{pmatrix}
   \sup_{\xi \in \cU} f_1(x,\xi)\\
   \sup_{\xi \in \cU} f_2(x,\xi)
 \end{pmatrix}
&&
\end{align*}
Recall that the functions $f_1$ and $f_2$ are linear in $x$ for every fixed $\xi \in \cU$
and $\cX=P \cap (\Z^{k} \times \R^{n-k})$ for a polyhedron $P$ and $0 \leq k \leq n$, i.e., without the supremum \eqref{eq:GrundproblemBRO} would satisfy the requirements of Lemma~\ref{lem:Dichotomic}.
However, since the functions $f_i^\cU \colon \cX \to \R$, $x \mapsto \sup_{\xi \in \cU} f_i(x,\xi)$, $i=1,2$, are not linear, 
we aim to transform \eqref{eq:GrundproblemBRO} 
to a biobjective mixed-integer linear optimization problem,
i.e., to a problem of type \eqref{eq:bioobjective-linear-mip} for which we can
  apply dichotomic search.

We proceed in two steps. The first step is to transform \eqref{eq:GrundproblemBRO} to its bottleneck
version, i.e., to 
\begin{align*}
\textup{BRO}_{\textup{BN}}(\cU) 
&& 
\min \begin{pmatrix} y_1 \\ y_2\end{pmatrix}
&& \\
&&\text{s.t. }
y_1 &\geq f_1(x,\xi) 
& \; \forall \xi \in \cU \\
&&
y_2 &\geq f_2(x,\xi) 
& \; \forall \xi \in \cU \\
  && x &\in \cX & \\
  &&  y &\in \R^2 & 
\end{align*}
This is justified by the following lemma which regards the relationship of $\textup{BRO}(\cU)$ and $\textup{BRO}_{\textup{BN}}(\cU)$.

\begin{lemma} \label{lem:BRO-BROBN}
Let a problem of type \eqref{eq:GrundproblemBRO} be given.
In particular, let  $\cU$ be compact and $f_i(x,\cdot) \colon \conv(\cU) \to \R$, $x \in \cX$, $i=1,2$, be continuous.
Then
\begin{enumerate}[label=(\roman*)]
\item $\{ (x,y): x \in \cX, y \geq f^{\cU}(x) \} \not= \emptyset$
  is the set of feasible solutions for $\textup{BRO}_\textup{BN}(\cU)$.
\item $X \subseteq \cX$ is the set of efficient solutions to $\textup{BRO}(\cU)$
  if and only if $\{ (x,y): x \in X, y = f^{\cU}(x) \}$ is the set of efficient solutions to
  $\textup{BRO}_\textup{BN}(\cU)$.
  In particular, the set of nondominated points for $\textup{BRO}(\cU)$ and
  $\textup{BRO}_\textup{BN}(\cU)$
  coincide.
\item 
  The set of extreme supported nondominated points for $\textup{BRO}(\cU)$ and
  $\textup{BRO}_\textup{BN}(\cU)$ coincide. 
\item $X \subseteq \cX$ is a representative set of extreme supported efficient solutions to $\textup{BRO}(\cU)$ if and only if $\{ (x,y) \colon x \in X, y = f^{\cU}(x) \}$   is a representative set of extreme supported  efficient solutions to
  $\textup{BRO}_\textup{BN}(\cU)$.
\end{enumerate}
\end{lemma}

\begin{proof} \
\begin{enumerate}[label=(\roman*)]
\item Directly by definition of $\textup{BRO}_\textup{BN}(\cU)$.
The feasible set of  $\textup{BRO}_\textup{BN}(\cU)$ is not empty due to compactness of $\cU$.

\item Let $(x,y)$ be efficient for $\textup{BRO}_\textup{BN}(\cU)$.
We show that this yields $y=f^{\cU}(x)$: 
Clearly, $y \geq f^\cU(x)$ otherwise $(x,y)$ is not feasible for $\textup{BRO}_\textup{BN}(\cU)$, (see (i)).
Now assume that $y_i > \max_{\xi \in \cU} f_i(x,\xi)$ for $i \in \{1,2\}$.
Then $(x,y)$ is dominated by the feasible solution $(x,f^\cU(x))$ and hence not efficient.
The set of efficient solutions to  $\textup{BRO}_\textup{BN}(\cU)$ hence is contained in $\{(x,f^{\cU}(x)): x \in \cX\}$. 

Note that $f^{\cU}(x)$ is the objective function value of $x$ in $\textup{BRO}(\cU)$ and also of $(x,f^{\cU}(x))$ in $\textup{BRO}_\textup{BN}(\cU)$.
This yields that $x$ is efficient to $\textup{BRO}(\cU)$ if and only if $(x,f^{\cU}(x))$ is efficient to $\textup{BRO}_\textup{BN}(\cU)$.
Hence, $X$ is the set of efficient solutions to $\textup{BRO}(\cU)$ if and only if $\{(x,f^{\cU}(x)): x \in X\}$ is  the set of efficient solutions to $\textup{BRO}_\textup{BN}(\cU)$ and the sets of nondominated points of both problems coincide.

\item The definition of extreme supported nondominated solutions only uses the set of nondominated points in objective space.
Due to (ii) the set of nondominated points for $\textup{BRO}(\cU)$ and $\textup{BRO}_\textup{BN}(\cU)$ coincide, hence also their extreme supported nondominated points. 

\item Let $X \subseteq \cX$ be a representative set of extreme supported efficient solutions to $\textup{BRO}(\cU)$.
Then $f^{\cU}(X)$ is the set of extreme supported nondominated points for $\textup{BRO}(\cU)$.
According to (iii), $f^{\cU}(X)$ is also the set of extreme supported nondominated points to $\textup{BRO}_\textup{BN}(\cU)$.
Since $f^{\cU}(X)$ is the image of $\{(x,f^{\cU}(x)): x \in X\}$ for $\textup{BRO}_\textup{BN}(\cU)$, the latter set is a representative set of extreme supported efficient solutions to $\textup{BRO}_\textup{BN}(\cU)$.
   
Let there be a representative set of extreme supported efficient solutions to $\textup{BRO}_\textup{BN}(\cU)$.
By (ii), it takes the form $\{ (x,y) \colon x \in X, y = f^{\cU}(x) \}$ for some $X \subseteq \cX$.

Its image $f^{\cU}(X)$ then is the set of extreme supported nondominated solutions to $\textup{BRO}_\textup{BN}(\cU)$, and according to (iii), also to $\textup{BRO}(\cU)$. Consequently, $X$ is a representative set of extreme supported efficient solutions to $\textup{BRO}(\cU)$.
\end{enumerate}
\end{proof}

$\textup{BRO}_\textup{BN}(\cU)$ has linear objective functions. 
However, to ensure its feasible set meets the requirements of Lemma~\ref{lem:Dichotomic},
we additionally need that the feasible set of $\textup{BRO}_\textup{BN}(\cU)$ is a polyhedron intersected with $\Z^k \times \R^{n-k}$ for $0 \leq k \leq n$. 
Then Algorithm~\ref{algo:DichotomicSearch} can be applied to $\textup{BRO}_\textup{BN}(\cU)$ and determines all its extreme supported nondominated points and a representative set of extreme supported efficient solutions.
In the following lemma we show more, namely that we do not need the bottleneck version but can apply Algorithm~\ref{algo:DichotomicSearch} directly to $\textup{BRO}(\cU)$ to receive the extreme supported nondominated points and a representative set of extreme supported efficient solutions of $\textup{BRO}(\cU)$, if the set $\cU$ of scenarios is finite.

\begin{lemma} \label{lem:DichotomicSearch-BRO-finiteU}
Let a problem of type \eqref{eq:GrundproblemBRO} be given and let (BRO-1) and (BRO-3) hold.
We assume that $\cU$ is non-empty and finite.

Then Algorithm~\ref{algo:DichotomicSearch} applied to \eqref{eq:GrundproblemBRO}
returns a set $Y^\ast$ containing all extreme supported nondominated
points and a set $X^\ast$ containing a representative set of extreme supported efficient solutions
after $2|Y^\ast| - 3$ iterations (lines 14-23) if $|Y^\ast| > 2$ and zero iterations if $|Y^\ast|=1$.
\end{lemma}
\begin{proof}
The proof is in two parts: 
First, we show that dichotomic search \emph{applied to the bottleneck version}  $\textup{BRO}_\textup{BN}(\cU)$ of $\textup{BRO}(\cU)$ returns a representative set of extreme supported efficient solutions and the set of all extreme supported nondominated points for the (non-bottleneck) problem $\textup{BRO}(\cU)$. 
Second, we show that applying dichotomic search directly to $\textup{BRO}(\cU)$ yields the exact same solutions as applying it to the bottleneck version $\textup{BRO}_\textup{BN}(\cU)$.

For the first part we use that the bottleneck version of the problem,
i.e., $\textup{BRO}_\textup{BN}(\cU)$, meets the requirements of Lemma~\ref{lem:Dichotomic}:
We use the assumptions made for \eqref{eq:GrundproblemBRO} and see that
$\textup{BRO}_\textup{BN}(\cU)$ is
a biobjective problem with two linear objectives $y_1$ and $y_2$.
For the feasible set note that the original feasible set $\cX$ of $\textup{BRO}(\cU)$ 
is given as $\cX=P \cap (\R^{n-k} \times \Z^k)$. Since we add two variables and two
linear constraints for each scenario from the finite set $\cU$ (see part (i) of
Lemma~\ref{lem:BRO-BROBN}) also the resulting feasible set for $\textup{BRO}_\textup{BN}(\cU)$ can be written as $P^\prime \cap (\R^{n^\prime-k} \times \Z^k)$ with a new polyhedron $P^\prime$ and dimension $n^\prime=n+2$.
Furthermore, \eqref{eq:ideal-point-existent} holds due to Corollary~\ref{cor:PropertiesForBRO}.

Thus, due to Lemma~\ref{lem:Dichotomic},
dichotomic search (Algorithm~\ref{algo:DichotomicSearch}) can be applied
and a set $Y^\ast_{BN}$ containing all extreme supported nondominated points 
and a representative set of extreme supported efficient solutions $X^\ast_{BN}$ for $\textup{BRO}_\textup{BN}(\cU)$ are determined after $2|Y^\ast| - 3$ iterations (lines 14-23) if $|Y^\ast| > 2$ and zero iterations if $|Y^\ast|=1$.

Lemma~\ref{lem:BRO-BROBN}~(iv) shows that 
$X^*_{BN}=\{(x,f^{\cU}(x)) \colon x \in X \}$ for some set $X \subseteq \cX$ which is a representative set of extreme supported efficient solutions of $\textup{BRO}(\cU)$.

For the second part, note that the difference between using
$\textup{BRO}(\cU)$ or $\textup{BRO}_{BN}(\cU)$ concerns
lines 2, 3, 5, 6, and each iteration of line 17 of Algorithm~\ref{algo:DichotomicSearch}.
However, there is no difference between applying these steps to 
$\textup{BRO}(\cU)$ and $\textup{BRO}_\textup{BN}(\cU)$: the feasible set of the latter problem is of higher dimension than the feasible set of the former but their outcomes in the objective space $\R^2$ coincide (see Lemma~\ref{lem:BRO-BROBN}) and only those are needed for subsequent computations.
\end{proof}

The lemma above justifies the application of dichotomic search to our problem of interest \eqref{eq:GrundproblemBRO} if $\cU$ is finite.
However, in  \eqref{eq:GrundproblemBRO} $\cU$ may be a polytope. On the other hand,
in (BRO-3) we
made the additional -- and thus far unnecessary -- assumption that
$f_i(x, \cdot) \colon \cU \to \R$, $i=1,2$ are quasi-convex.
Utilizing this additional requirement, we now show that
Lemma~\ref{lem:DichotomicSearch-BRO-finiteU} is still valid if $\cU$ is a polytope
instead of a finite set.

\begin{lemma} \label{lem:DichotomicSearch-BRO-polytopeU}
Let a problem of type \eqref{eq:GrundproblemBRO} be given and let
(BRO-1) and (BRO-3) hold.
We assume that $\cU$ is a polytope.

Then Algorithm~\ref{algo:DichotomicSearch} applied to \eqref{eq:GrundproblemBRO}
returns a set $Y^\ast$ containing all extreme supported nondominated
points and a set $X^\ast$ containing a representative set of extreme supported efficient solutions
after $2|Y^\ast| - 3$ iterations (lines 14-23) if $|Y^\ast| > 2$ and zero iterations if $|Y^\ast|=1$.
\end{lemma}
\begin{proof}
If $\cU$ is a polytope it has a finite number of (not necessarily known) extreme points $\xi_1,\ldots,\xi_l$.
Since the functions $f_1(x,\cdot), f_2(x,\cdot) \colon \conv(\cU) \to \R$, $x \in \cX$ are quasi-convex, according to \cite[Theorem~5.9]{EIS13}, $\textup{BRO}(\cU)$ and $\textup{BRO}(\{\xi_1,\ldots,\xi_l\})$ are equivalent since their objective functions $f^\cU$ and $f^{\{ \xi_1, \ldots, \xi_l \}}$ are the same.

Lemma~\ref{lem:DichotomicSearch-BRO-finiteU} justifies that we can apply Algorithm~\ref{algo:DichotomicSearch} to $\textup{BRO}(\{\xi_1,\ldots,\xi_l\})$ and get all extreme supported nondominated points and a representative set of extreme supported efficient solutions of $\textup{BRO}(\{\xi_1,\ldots,\xi_l\})$ and hence also of $\textup{BRO}(\cU)$ in 
$2|Y^\ast| - 3$ iterations if $|Y^\ast| > 2$ and zero iterations if $|Y^\ast|=1$. 
This, however, requires that $\xi_1,\ldots,\xi_l$ are known.
Since finding the vertices of a given polytope, known as \emph{vertex enumeration}, is a hard problem (see \cite{khachiyan2009generating}),
we apply Algorithm~\ref{algo:DichotomicSearch} directly to $\textup{BRO}(\cU)$ without using the extreme points of $\cU$.
Luckily, this can be done by using the equivalence of $\textup{BRO}(\{\xi_1,\ldots,\xi_l\})$ and $\textup{BRO}(\cU)$ once more:

Namely, we replace $\textup{BRO}(\{\xi_1,\ldots,\xi_l\})$ by $\textup{BRO}(\cU)$ whenever it occurs in Algorithm~\ref{algo:DichotomicSearch}, i.e, in Steps~2,3,6,7 and in Step~17 and note that it does not change any result.
Summarizing, we can also apply Algorithm~\ref{algo:DichotomicSearch} directly to $\textup{BRO}(\cU)$.
\end{proof}

\section{Optimization-pessimization for biobjective optimization} \label{sec:OptPess}
In the previous section, we conceived the problem \eqref{eq:GrundproblemBRO} primarily as a biobjective problem -- with the more complicated objective function $f^\cU$ -- and suggested biobjective optimization methods. In this section, we take the perspective of a robust optimizer and apply a method known from robust optimization.
More precisely, we use a cutting plane approach, called \emph{optimization-pessimization}, which is designed to find minmax robust solutions of uncertain (but single-objective)
optimization problems. 
The approach is reviewed in Section \ref{sec:OptPess-intro} and extended
to multi-objective optimization problems in Section~\ref{sec:OptPess-multiobjective}.

\subsection{Optimization-pessimization for single-objective robust optimization} \label{sec:OptPess-intro}

This section deals with uncertain (single-objective) optimization problems, 
\begin{equation*} \tag{\ref{eq:UncertainProblem} revisited}
\left\lbrace
\min_{x\in\cX} h(x,\xi) \colon \xi \in \cU
\right\rbrace
\text.
\end{equation*}
More specifically, we want to determine minmax robust solutions for such problems and, to that end, solve the robust counterpart,
\begin{align*} \tag{\ref{eq:MinSupProblem} revisited}
\PsingleU{\cU}
&&
\min_{x \in \cX} \sup_{\xi \in \cU} h(x,\xi)
\text.
&&
\end{align*}
We assume that for every fixed $x \in \cX$ the function $h(x,\cdot) \colon \conv(\cU) \to \R$ is continuous and quasi-convex and that $\cU$ is compact.
Hence, $\sup_{\xi \in \cU} h(x,\xi)$ is attained for all $x \in \cX$ and from now on we can write $\max_{\xi \in \cU} h(x,\xi)$ instead. Let us denote 
$z(\cU) \coloneqq \min_{x \in \cX} \max_{\xi \in \cU} h(x,\xi)$
as optimal objective function value of \eqref{eq:MinSupProblem} for a given uncertainty set $\cU$.
\medskip

There exist many approaches for solving problem \eqref{eq:MinSupProblem}, which are grouped in \cite{HertogEtAl2015} into two classes:
The first class of algorithms is based on reformulations to avoid the maximum over an (often infinite) set. 
We follow this approach in Section~\ref{sec:dualize}.
The algorithms of the second class proceed iteratively.
They start with a small set of scenarios and add scenarios step by step.
These approaches are known under various names such as 
\emph{cutting set method} (\cite{MutapicBoyd09}),
\emph{cutting plane method} (\cite{bertsimas2016reformulation}), 
\emph{scenario relaxation procedure} (\cite{assavapokee2008scenario}, \cite{AisBazVanSurvey09}),
\emph{outer approximation method} (\cite{reemtsen1994some} \cite{burger2013polyhedral} \cite{GoeSchoe13-AE}), 
\emph{(modified) Benders decomposition approach} (\cite{montemanni2006benders}, \cite{siddiqui2011modified}), or
\emph{implementor-adversarial framework} (\cite{bienstock2007histogram}).

We refer to it as \emph{optimization-pessimization}.
The idea is to utilize that robust optimization problems are easier to solve for (very) small uncertainty sets:
The routine starts with a reduced set of scenarios $\cU^\prime$ for which a robust solution is determined. For this solution, the routine determines a worst-case scenario out of the full uncertainty set $\cU$ which is added to $\cU^\prime$. 
For the new scenario set, a new robust solution is found.
This procedure is repeated until the quality of the solution found is good enough,  see Figure~\ref{fig:Opt-Pess-Scheme} for an illustration.

\begin{figure}[hbtp]
\centering
\usetikzlibrary{shapes.geometric,arrows}
\tikzstyle{block} = [rectangle, draw, text width=16em, text badly centered, rounded corners, minimum height=4em] 
\tikzstyle{emptyblock} = []  
\small
\begin{tikzpicture}[node distance=3.5cm, auto, arrow/.style={-latex, shorten >=1ex, shorten <=1ex, bend angle=45}]
\node [emptyblock]				(ctr) {};
\node [block, left of=ctr]		(opt)  {
	Optimization:\\ 
	Determine robust solution $x^\ast \in \cX$ of $\PsingleU{\cU^\prime}$ finding\\
	$x^\ast \in \argmin_{x \in \cX} \max_{\xi \in \cU^\prime} h(x,\xi)$
}; 
\node [block, right of=ctr]		(pess) {
	Pessimization:\\
	Determine worst-case scenario $\xi^\ast \in \cU$ for given $x^\ast$ finding\\
	$\xi^\ast \in \argmax_{\xi \in \cU} h(x^\ast,\xi)$};  
\node [block, below of=ctr, distance=1cm]		(add) {Add scenario: $\cU^\prime \coloneqq \cU^\prime \cup \{ \xi \}$};  
\draw [arrow, bend angle=45, bend left]  (opt) to (pess);
\draw [arrow, bend angle=30, bend left]  (pess) to (add);
\draw [arrow, bend angle=30, bend left]  (add) to (opt);
\end{tikzpicture}
\caption{Optimization-pessimization for robust single-objective optimization problems} \label{fig:Opt-Pess-Scheme}
\end{figure}
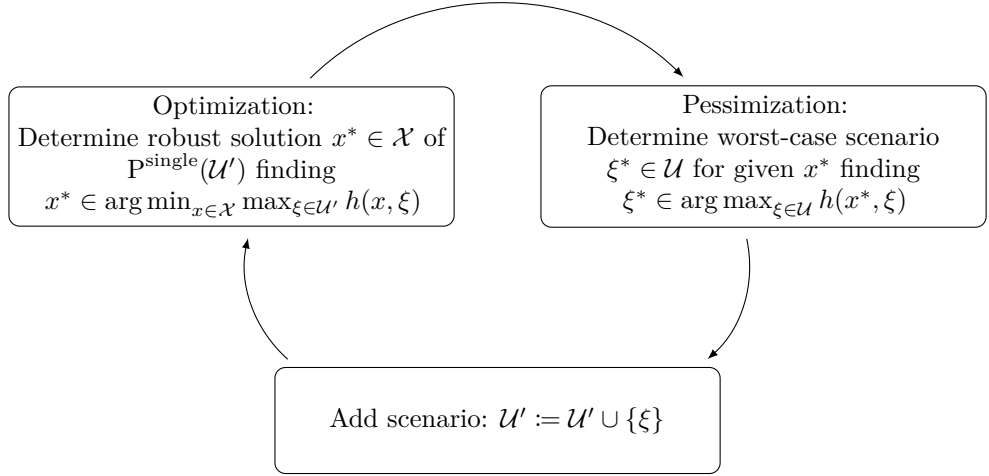

Formally, the optimization and pessimization problems are defined as follows:
For any $\cU^\prime \subseteq \cU$ the \emph{optimization problem}
is defined as
\begin{align*}
\PsingleU{\cU^\prime}
&&
z(\cU^\prime) \coloneqq \min_{x \in \cX} \max_{\xi \in \cU^\prime} h(x,\xi).
&&
\end{align*}
It is a relaxation of $\PsingleU{\cU}$ and, thus, yields a lower bound for $\PsingleU{\cU}$, i.e.,
\begin{equation}  \label{eq:lower}
  z(\cU^\prime) \leq z(\cU)
  \text.
\end{equation}

For a given $x \in \cX$, the \emph{pessimization problem} 
\begin{align*}
\text{Pess}(x)
&&
h^\cU(x) \coloneqq \max_{\xi \in \cU} h(x,\xi)
&&
\end{align*}
evaluates $x$ over the complete set of scenarios $\cU$ and, thus, provides an upper bound for $z(\cU)$, i.e.,
\begin{equation} \label{eq:upper}
  h^\cU(x) \geq z(\cU)
  \textup.
\end{equation}
Algorithm~\ref{algo:OptPess-single-objective}
describes how this method can be put to use algorithmically if $\cU$ is a polytope or finite.

\begin{algorithm}
    
\begin{algorithmic}[hbt!]
\small
	\REQUIRE{Robust optimization problem $P(\cU)$ as in \eqref{eq:MinSupProblem}.}
    \REQUIRE{Finite initial set $\cU^{(0)} \subseteq \cU$.}
	\ENSURE{Either $\cU$ finite or $\cU$ a polytope and $h(x,\cdot)$ continuous and quasi-convex.}
    \STATE{Set $k \coloneqq 0$.}
	\REPEAT%
		\STATE{Set $\cU^{(k+1)} \coloneqq \cU^{(k)}$.}
		
		\tikzmark{startopt41}
		\STATE{Determine $x^k \in \argmin_{x \in \cX} \{ \max_{\xi \in \cU^k} h(x,\xi) \}$. Set $z(\cU^{(k)}) \coloneqq \max_{\xi \in \cU^{(k)}} h(x^k,\xi)$.}
    	
    	\hfill~\tikzmark{endopt41}
    	
    	\tikzmark{startpess41}    	
		\STATE{For given $x^k$ determine solution $\xi^k \in \argmax_\cU h(x^k,\xi^k)$. Set $h^\cU(x^k) \coloneqq h(x^k,\xi^k)$.}  	
		\STATE{Add $\xi^k$ to $\cU^{(k+1)}$.}
		
	   	\hfill~\tikzmark{endpess41}
		
		\STATE{Set $k \coloneqq k+1$.}
                \UNTIL{$h^\cU(x^{k-1}) = z(\cU^{(k-1)})$.}
	\RETURN{robust solution $x^\ast$.}
	\RETURN{set of worst-case scenarios $\cU^\textup{FINAL}\coloneqq \cU^k$.}
\end{algorithmic}
\Textboxblue{startopt41}{endopt41}{Optimization}
\Textboxblue{startpess41}{endpess41}{Pessimization}
\caption{Optimization-pessimization for single-objective robust optimization} \label{algo:OptPess-single-objective}
\end{algorithm}

The routine produces a sequence of sets
\begin{equation}
\cU^{(0)} \subseteq \cU^{(1)} \subseteq \cU^{(2)} \subseteq \dots \subseteq \cU.
\end{equation}
According to \eqref{eq:lower} we receive a sequence of lower bounds 
\begin{equation} \label{eq:OptPess-LBSequence}
z(\cU^{(0)}) \leq  z(\cU^{(1)}) \leq z(\cU^{(2)}) \leq \dots \leq  z(\cU)
\end{equation}
and, a feasible solution $x^k$ in each iteration from which we can
  derive an upper bound according to \eqref{eq:upper}, i.e.,
\begin{equation} \label{eq:bounds-kth-iteration} 
z(\cU^{(k)}) \leq z(\cU) \leq h^\cU(x^k)
\textup.
\end{equation}
We stop when lower and upper bound coincide.
Then an optimal solution to \eqref{eq:MinSupProblem} and thus a (minmax) robust optimal solution to \eqref{eq:UncertainProblem} has been found.
For more detailed discussions of the method we refer to \cite{bertsimas2016reformulation,AisBazVanSurvey09,PaeSchoe20}.
The finiteness of Algorithm~\ref{algo:OptPess-single-objective}  for uncertainty sets $\cU$ that are polytopes is shown in the following lemma in part (ii).

\begin{lemma} \label{lem:OptPessSingleObjective}
Assume that $\PsingleU{\cU}$ has an optimal solution and $\PsingleU{\cU^\prime}$ has an optimal solution for all finite $\cU^\prime \subseteq \cU$.
\begin{enumerate}[label=(\roman*)]
\item 
Let $\cU$ be finite. 
Then Algorithm~\ref{algo:OptPess-single-objective} returns a solution to
$\PsingleU{\cU}$
in at most $|\cU|$ iterations.
\item  
Let $\cU$ be a polytope or finite and let $\ext(\cU)$ be its set of extreme points.
Let $h(x,\cdot) \colon \conv(\cU) \to \R$, $x \in \cX$, be continuous and quasi-convex.
Then Algorithm~\ref{algo:OptPess-single-objective} returns a solution to $\PsingleU{\cU}$ in at most $|ext(\cU)|$ iterations if we choose an algorithm for the pessimization problem $\textup{Pess}(x)$ which always finds an extreme point of $\cU$.
\end{enumerate}
\end{lemma}

\begin{proof}
Algorithm~\ref{algo:OptPess-single-objective} stops if the lower and upper bound for $z(\cU)$ coincide (see line 8 of Algorithm~\ref{algo:OptPess-single-objective}), i.e., if $h^\cU(x^k)=z(\cU^{(k)})$.
We hence have that $x^k$ is an optimal solution.
Note that
\begin{equation} \label{helper1}
\max_{\xi \in \cU}h(x^k,\xi)
=h^\cU(x^k)
=z(\cU^{(k)})
=\max_{\xi \in \cU^{(k)}}h(x^k,\xi)
\textup,
\end{equation}
if at least one worst-case scenario of $\cU$ for $x^k$ is already contained in $\cU^{(k)}$.
For a finite uncertainty set, in every iteration either a new worst-case scenario is added or \eqref{helper1} holds and the procedure stops.
The latter happens after at most $|\cU|$ iterations which shows (i).


For (ii), consider the pessimization problem $\textup{Pess}(x^k)$: 
here we maximize a continuous function over a compact set $\cU$, i.e., a maximum always exists.
Since $h(x,\cdot)$ is quasi-convex, a maximum is always attained at an extreme point of $\cU$.
If we choose an algorithm that returns an extreme point for such optimization problems, we add a new extreme point in each iteration.
Since the number of extreme points of $\cU$ is finite the procedure stops when \eqref{helper1} holds. As in part (i) this happens after at most $|\ext(\cU)|$ iterations.
 \end{proof}

We remark that Algorithm~\ref{algo:OptPess-single-objective} also converges for bounded non-polyhedral sets $\cU$ under uniform Lipschitz-continuity in $x$ for all fixed values of $\xi$ (see
\cite{MutapicBoyd09}).

\subsection{Optimization-pessimization for multi-objective robust optimization} \label{sec:OptPess-multiobjective}

\paragraph{Optimization and pessimization problem in the multiobjective case.}
In order to apply optimization-pessimization to \eqref{eq:GrundproblemBRO}, 
we need to generalize it to biobjective problems. In this section we go a step further
  and consider minmax problems with $p$ objective functions, i.e.,
\begin{align*} \tag{\ref{eq:Grundproblem-p-kriteriell} revisited} 
P(\cU)
&&
\min_{x\in\cX}
\begin{pmatrix}
\sup_{\xi \in \cU} f_1(x,\xi)\\
\sup_{\xi \in \cU} f_2(x,\xi)\\
\vdots \\
\sup_{\xi \in \cU} f_p(x,\xi)\\
\end{pmatrix}
\text.
&&
\end{align*}
for which we aspire to determine a representative set of extreme supported efficient solutions.
With this purpose in mind, we develop a generalized version of optimization-pessimization of Section~\ref{sec:OptPess-intro}.

The optimization problem $P(\cU^\prime)$ for $\cU^\prime \subseteq \cU$ is
the multiobjective optimization problem
\begin{align}
P(\cU^\prime)
&&
z(\cU^\prime)
\coloneqq
\min_{x \in \cX}
\begin{pmatrix}
\sup_{\xi \in \cU^\prime} f_1(x,\xi) \\
\sup_{\xi \in \cU^\prime} f_2(x,\xi) \\
\vdots \\
\sup_{\xi \in \cU^\prime} f_p(x,\xi)
\end{pmatrix}
\text. \label{eq:P_U'}
&&
\end{align}
The pessimization problem
\begin{align}
\textup{Pess}(x)
&&
f^\cU(x)
\coloneqq
\begin{pmatrix}
\sup_{\xi \in \cU} f_1(x,\xi) \\
\sup_{\xi \in \cU} f_2(x,\xi) \\
\vdots \\
\sup_{\xi \in \cU} f_p(x,\xi)
\end{pmatrix} \label{eq:Pess(x)}
&&
\end{align}
for given $x \in \cX$ consists of $p$ indepedent pessimization problems.

\paragraph{Lower and upper bounds provided by the optimization and the pessimization problem.}
We first discuss the optimization and pessimization problems in relation to \eqref{eq:Grundproblem-p-kriteriell} which we are interested to solve. 

For single-objective problems \eqref{eq:Grundproblem-p-kriteriell}, the solutions to $\PsingleU{\cU^\prime}$ and $\textup{Pess}(x)$ provide lower and upper bounds to \eqref{eq:Grundproblem-p-kriteriell}.
In the multi-objective setting we do not evaluate single solutions,
but we need to evaluate (Pareto) sets.
Sets can be compared by set order relations, one
of the most common ones is the \emph{upper setless order}: For two sets $Y_1,Y_2 \subset \R^p$
it is defined as follows:
\begin{equation*}
Y_1 \preceq^{upp} Y_2 \mbox{ if for all } y \in Y_2 \mbox{ there exists } \tilde{y}
  \in Y_1 \mbox{ with } \tilde{y} \leq y.
\end{equation*}
In this sense, we can say that $Y_1$ is an (upper setless) lower bound on $Y_2$.
We now use the upper setless order to generalize \eqref{eq:bounds-kth-iteration}
showing that for multi-objective
optimization we also get lower and upper bounds on \eqref{eq:Grundproblem-p-kriteriell}
when solving \eqref{eq:P_U'} and \eqref{eq:Pess(x)} for a subset $\cU^\prime$ of $\cU$.
More precisely, let $X^\ast(\cU)$ be the set of efficient solutions to \eqref{eq:Grundproblem-p-kriteriell}. 
Then $\{ f^{\cU}(x) \colon x \in X^\ast(\cU) \}$ describes the
Pareto frontier of \eqref{eq:Grundproblem-p-kriteriell}.
It can be bounded based on the solutions of the relaxation $P(\cU^\prime)$ as follows.

\begin{lemma} \label{lem:upper-setless-bounds}
  Let $\cU^\prime \subseteq \cU$ and denote $X^\ast(\cU^\prime)$ and
  $X^\ast(\cU)$ the set of efficient solutions of
  $P(\cU^\prime)$, and $P(\cU)$, respectively.
  Assume that $P(\cU^\prime)$, and $P(\cU)$ both satisfy the domination
  property \eqref{eq:domination-property}.
  Then the following holds for the upper setless order $\preceq^{upp}$:
  \begin{equation}
    \label{helper2}
    \{f^{\cU^\prime}(x)\colon x \in X^\ast(\cU^\prime)\} \preceq^{upp}
    \{f^{\cU}(x) \colon x \in X^\ast(\cU)\} \preceq^{upp}
    \{f^{\cU}(x)\colon x \in X^\ast(\cU^\prime)\} 
  \end{equation}
\end{lemma}


\begin{proof}
  We first show the left hand side of \eqref{helper2}. To this end, take
  $x \in X^\ast(\cU)$. We want to show that there exists $\tilde{x} \in
  X^\ast(\cU^\prime)$ such that
  \begin{equation}
    \label{helper3}
    f^{\cU^\prime}(\tilde{x}) \leq f^{\cU}(x).
    \end{equation}
  From $\cU^\prime \subseteq \cU$ we get that $f^{\cU^\prime}(x) \leq f^{\cU}(x)$,
  see \eqref{eq:U'relaxation}. Hence, if $x \in X^\ast(\cU^\prime)$ we set $\tilde{x} \coloneqq x$
  and are done. Otherwise, $x \not\in X^\ast(\cU^\prime)$, i.e., $x$ is not an efficient
  solution to $P(\cU^\prime)$. Then, due to the domination property, there
  exists $\tilde{x} \in X^\ast(\cU^\prime)$ with
  $f^{\cU^\prime}(\tilde{x}) \leq f^{\cU^\prime}(x) \leq f^{\cU}(x)$ and \eqref{helper3}
  holds.

  For the right hand side, we take $x \in X^\ast(\cU^\prime)$. The goal is to find
  $\tilde{x} \in X^\ast(\cU)$
  such that
  \begin{equation*}
  f^{\cU}(\tilde{x}) \leq f^{\cU}(x)
  \textup.
  \end{equation*}
  Similar as above, if $x \in X^\ast(\cU)$ we set $\tilde{x} \coloneqq x$ and are done. Otherwise,
  $x$ is not efficient for $P(\cU)$ and due to the domination property we find
  $\tilde{x} \in X^\ast(\cU)$ with $f^\cU(\tilde{x}) \leq f^\cU(x)$ which finishes
  the proof.
\end{proof}

The statement in~\eqref{helper2} is the multiobjective analog of~\eqref{eq:bounds-kth-iteration}.

\paragraph{Reduction of the scenario set.} 
In this paragraph, we examine the conditions under which a reduced uncertainty set $\cU^\prime \subset \cU$ already contains all relevant scenarios, such that the efficient solutions $P(\cU^\prime)$ are the same as those of $P(\cU)$.
In the single-objective setting this is the case if for an efficient solution to $P(\cU^\prime)$ a worst-case scenario is already included in $\cU^\prime$ (see \eqref{helper1}).
We call this the worst-case property \eqref{eq:WCincluded}.
The following theorem formalizes the above considerations and shows when the efficient solutions of $P(\cU)$ and $P(\cU^\prime)$ coincide.

  
\begin{thm} \label{thm:Konvergenzaussage}
Let $\cU^\prime \subset \cU$. 
Consider $x \in \cX$. 
If we have 
\leqnomode 
\begin{align} \tag{wc} \label{eq:WCincluded}
\sup_{\xi \in \cU^\prime} f_i(x,\xi) = \sup_{\xi \in \cU} f_i(x,\xi) \text{ for all } i=1,2,\dots,p \text{,}
\end{align}
\reqnomode
then the following holds: 
\begin{equation*}
x \text{ is efficient for } P(\cU^\prime) \Rightarrow x \text{ is efficient for } P(\cU) \text.
\end{equation*}

Additionally, if the domination property \eqref{eq:domination-property} holds for $P(\cU^\prime)$ and \emph{all} solutions $x \in \cX$ that are efficient for $P(\cU^\prime)$ satisfy \eqref{eq:WCincluded}, then the following holds:
\begin{align*}
x \text{ is  efficient for } P(\cU^\prime) \Leftrightarrow x \text{ is efficient for } P(\cU) \text.
\end{align*}
\end{thm}

\begin{proof}
\underline{$\Rightarrow$:}
Let $x$ be efficient for $P(\cU^\prime)$ and satisfy \eqref{eq:WCincluded}, i.e., $f^{\cU^\prime}(x)=f^\cU(x)$.
Assume to the contrary that $x$ is not efficient for $P(\cU)$, i.e., there exists $x^\prime \in \cX$, such that
\begin{equation}
	f^\cU(x^\prime)
	\preceq
	f^\cU(x)
	\text.
	\label{eq:pb-proof01}
\end{equation}
$\cU^\prime \subseteq \cU$, hence $f^{\cU^\prime}(x^\prime) \leq f^{\cU}(x^\prime)$, see \eqref{eq:U'relaxation}. This leads to
\[
	f^{\cU^\prime}(x^\prime)
       	\stackrel{\eqref{eq:U'relaxation}}{\leq}
	f^{\cU}(x^\prime)
	\stackrel{\eqref{eq:pb-proof01}}{\preceq}
	f^{\cU}(x)
	\stackrel{\eqref{eq:WCincluded}}{=}
	f^{\cU^\prime}(x)
	,
\]
which contradicts efficiency of $x$ for $P(\cU^\prime)$.

\underline{$\Leftarrow$:}
Let \eqref{eq:WCincluded} hold for all solutions which are efficient for $P(\cU^\prime)$ and let $x$ be efficient for $P(\cU)$.
Assume to the contrary that $x \in \cX$ is not efficient for $P(\cU^\prime)$.
Then, since the domination property holds, 
there is a solution $x^\prime \in \cX$ that is efficient for P($\cU^\prime$) such that
\begin{equation}
	f^{\cU^\prime}(x^\prime)
	\preceq
	f^{\cU^\prime}(x)
	\text.
	\label{eq:pb-proof02}
\end{equation}
Note that since $x^\prime$ is efficient for $P(\cU^\prime)$, it satisfies \eqref{eq:WCincluded}. Together with $\cU^\prime \subseteq \cU$ we receive
\begin{equation*}
	f^{\cU}(x^\prime)
	\stackrel{\eqref{eq:WCincluded}}{=}
	f^{\cU^\prime}(x^\prime)
	\stackrel{\eqref{eq:pb-proof02} }{\preceq}
	f^{\cU^\prime}(x)
  	\stackrel{\eqref{eq:U'relaxation}}{\leq}
	f^{\cU}(x)
	\text.
\end{equation*}
This contradicts the assumption of $x$ being efficient for $P(\cU)$.
\end{proof}

Checking \emph{all} efficient solutions of a multiobjective problem is computationally hard (or even impossible).
Thus, in the next result we strengthen the above theorem in a fashion that \eqref{eq:WCincluded} must only be satisfied for all solutions from a representative set.

\begin{thm} \label{thm:Konvergenzaussage_verstaerkt}
Let the domination property \eqref{eq:domination-property} be satisfied for $P(\cU)$ and $P(\cU^\prime)$.
If there is a representative set $\Rep^\prime$ of efficient solutions for $P(\cU^\prime)$ whose elements satisfy \eqref{eq:WCincluded}, then we have:
\begin{enumerate}[label=(\roman*)]
\item
$x \in \Rep^\prime$ $\Rightarrow$ $x$ is efficient for  $P(\cU)$,
\item
$x$ is efficient for $P(\cU)$ $\Rightarrow$ $x$ is efficient for  $P(\cU^\prime)$, and
\item
$\Rep^\prime$ is a representative set of efficient solutions to $P(\cU)$.
\end{enumerate}
\end{thm}

\begin{proof} \
\begin{enumerate}[label=(\roman*)]
\item Let $x \in \Rep^\prime$. 
In particular, $x$ is efficient for $P(\cU^\prime)$ and by assumption it satisfies \eqref{eq:WCincluded}. 
We can hence apply Theorem~\ref{thm:Konvergenzaussage} and conclude that $x$ is efficient for $P(\cU)$.    
    
\item
Let $x$ be efficient for $P(\cU)$ and assume $x$ is not efficient for $P(\cU^\prime)$. 
Due to the domination property, there is
is a solution $x^\prime$ that satisfies $ f^{\cU^\prime}(x^\prime) \preceq f^{\cU^\prime}(x)$. Moreover, since $\Rep^\prime$ is a representative set for $P(\cU^\prime)$ we can choose $x^\prime \in \Rep^\prime$. Hence,
\eqref{eq:WCincluded} holds for $x^\prime$ and we receive

\begin{equation*}
	f^{\cU}(x^\prime)
	\stackrel{\eqref{eq:WCincluded}}{=}
	f^{\cU^\prime}(x^\prime)
	\stackrel{}{\preceq}
	f^{\cU^\prime}(x)
  	\stackrel{\eqref{eq:U'relaxation}}{\leq}
	f^{\cU}(x)
	\text.
\end{equation*}
This contradicts efficiency of $x$ for $P(\cU)$.
\item
Let $\Rep \subset \cX$ be a representative set of efficient solutions for $P(\cU)$.
We show that $f^\cU(\Rep^\prime) = f^\cU(\Rep)$.
\\
\underline{$\subset$:}
Let $y^\prime \in f^\cU(\Rep^\prime)$.
Then $y^\prime=f^\cU(x^\prime)$ for some $x^\prime \in \Rep^\prime$.
According to (i), $x^\prime$ is efficient for $P(\cU)$, hence 
$y^\prime \in f^\cU(\Rep)$.
\\
\underline{$\supset$:}
Let $y \in f^\cU(\Rep)$.
Then $y=f^\cU(x)$ for some $x$ that is efficient for $P(\cU)$.
According to (ii), $x$  is also efficient for $P(\cU^\prime)$.
Hence, $x^\prime \in \Rep^\prime$ exists such that
$f^{\cU^\prime}(x) = f^{\cU^\prime}(x^\prime)$.
This leads to
\[
y=f^\cU(x)
\stackrel{\eqref{eq:U'relaxation}}{\geq}
f^{\cU^\prime}(x)
=
f^{\cU^\prime}(x^\prime)
\stackrel{\eqref{eq:WCincluded}}{=}
f^{\cU}(x^\prime)
\textup.
\]
Since by assumption $y$ is nondominated for $P(\cU)$, equality must hold true.
Thus, $y=f^{\cU}(x^\prime)$ for $x^\prime \in \Rep^\prime$ and, consequently,
$y \in f^\cU(\Rep^\prime)$.
\end{enumerate}
\end{proof}

Theorem~\ref{thm:Konvergenzaussage_verstaerkt} shows that it is 
not necessary to check the worst-case property~\eqref{eq:WCincluded} for
\emph{all} efficient solutions, rather it is sufficient to check it only
for a representative set.
However, a representative set may still be infinite, even for linear problems.

In Theorem~\ref{thm:Konvergenzaussage_verstaerkt-esn} we show that
the statement of Theorem~\ref{thm:Konvergenzaussage_verstaerkt} remains valid
if we replace the set of all efficient solutions not only by a representative set, but even by a
representative set of only their extreme supported solutions.
Recall that in Section~\ref{sec:DichotomicSearch-intro} with dichotomic search we provided an algorithm for computing such a representative set of extreme supported solutions.
In preparation for Theorem~\ref{thm:Konvergenzaussage_verstaerkt-esn} we need the following lemma and corollary that investigate the relation of extreme supported nondominated points with the set of all images $\cY=f^\cU(\cX)$.

\begin{lemma} \label{lem:WeaklyDomByNontrivConvCombOfESN}
Let a multi-objective optimization problem \eqref{eq:MultiobjectiveProblem} with $\cY \subsetneq \R^p$ compact be given and 
let $\cY_\textup{ESN} \not = \emptyset$ be its set of extreme supported nondominated points.
We assume that $\cY_\textup{ESN}$ is finite.
Then $\cY \subseteq \conv(\cY_\textup{ESN})+\R^p_{\geq}$ holds.
\end{lemma}
\begin{proof}
Assume there is a $\bar{y} \in \cY$ that does not lie in $\conv(\cY_\textup{ESN})+ \R^p$.
We then can show that there is also $\hat{y}$ outside of $\conv(\cY_\textup{ESN})+ \R^p$ which is extreme supported nondominated, a contradiction.

So, assume to the contrary that $\bar{y} \in \cY \setminus (\conv(\cY_\textup{ESN})+\R^p_{\geq})$ exists.
Then the sets $\{\bar{y}\}$
and 
$\conv(\cY_\textup{ESN}) + \R^p_{\geq1}$
are disjoint, nonempty, closed and convex sets.
Hence, a separating hyperplane exists (see \cite{boyd2004convex}),
i.e., $\nu \in \R^p \setminus \{0\}$ and $s \in \R$ exist such that
\begin{equation} \label{eq:separation}
\nu^t \bar{y} < s < \nu^t y, \; \forall y \in \conv(\cY_\textup{ESN})+\R^p_{\geq}
\textup{.}
\end{equation}

The elements of $\conv(\cY_\textup{ESN})+\R^p_{\geq}$ can get arbitrarily big in each component, 
hence $\nu_i \geq 0$ for all $i=1,2,\dots,p$.
Let now $z_\nu \coloneqq \min\{ \nu^t y \colon y \in \cY\}$ and $\cY_\nu 
\coloneqq  \argmin \{ \nu^t y \colon y \in \cY \}$.
Since $\bar{y} \in \cY$, we get
\begin{equation}
\nu^t y^{\ast} \leq \nu^t \bar{y}< s, \; \forall y^{\ast} \in \cY_\nu
.
\end{equation}
Together with \eqref{eq:separation} this shows that the elements in $\cY_\nu$ can be separated from $\conv(\cY_\textup{ESN})+\R_\geq^p$
and, hence, cannot be extreme supported nondominated themselves.
Specifically, the lexicographic minimum, $\hat{y} \coloneqq \lexmin_{y \in \cY_\nu}$, i.e., $\hat{y}_j = \min \{y_j \colon y \in \cY_\nu, y_1=\hat{y}_1, \dots, y_{j-1}=\hat{y}_{j-1}\}$, $j=1,2,\dots,p$, is not extreme supported nondominated (the existence of this point follows from compactness of $\cY$).

Hence, a nontrivial convex combination of nondominated points $y^{(1)},\dots,y^{(n)} \in \cY$ exists such that
\begin{equation} \label{eq:WeaklyDomByNontrivConvComb}
\hat{y} \geq \sum_{i=1}^n \lambda_i y^{(i)}
\textup.
\end{equation}

Now we assume that $y^{(i)} \not \in \cY_\nu$ for at least one $i=1,2,\dots,n$. Wlog., assume $y^{(1)} \not \in \cY_\nu$.
Then 
\begin{equation} \label{eq:WeaklyDomByNontrivConvComb_Nu}
\nu^t \sum_{i=1}^n \lambda_i y^{(i)}
= 
\lambda_1 \underbrace{\nu^t y^{(1)}}_{> z_\nu} 
+ \sum_{i=2}^n \underbrace{\nu^t \lambda_i y^{(i)}}_{\geq z_\nu}
>
\sum_{i=1}^n \lambda_i z_\nu
=
z_\nu
=
\nu^t \hat{y}
\textup.
\end{equation}
Since $\nu_i \geq 0$, $i=1,2,\dots,n$, \eqref{eq:WeaklyDomByNontrivConvComb_Nu} contradicts \eqref{eq:WeaklyDomByNontrivConvComb}. Thus, our assumption that $y^{(i)} \not \in \cY_\nu$ for at least one $i=1,2,\dots,n$ is contradicted
and we have that $y^{(i)} \in \cY_\nu$ for all $i=1,2,\dots,n$.

Consequently, a nontrivial convex combination only consisting of nondominated points $y^{(1)},\dots,y^{(n)} \in \cY_\nu \subseteq \cY$ exists such that \eqref{eq:WeaklyDomByNontrivConvComb} holds.
This, however, is not possible since, by definition, $\hat{y}$ is the lexicographic minimum of $\cY_\nu$
and thus all other elements of $\cY_\nu$ lie in the lexicographic cone
$\hat{y} + \{y \in \R^p \colon y_1=y_2 = \dots y_i=0, y_{i+1} > 0 \textup{ for some } i=0,1,\dots,p\}$.
\end{proof}

The following corollary will be used in the proof of the subsequent theorem.

\begin{corol} \label{cor:WeaklyDomByNontrivConvCombOfESN}
Under the assumptions of Lemma~\ref{lem:WeaklyDomByNontrivConvCombOfESN} for any $y \in \cY \setminus \cY_\textup{ESN}$, a nontrivial convex combination
\begin{equation*}
\sum_{i=1}^n \lambda_i y^{(i)} \leq y
\end{equation*}
with $y^{(1)}, \dots, y^{(n)} \in \cY_\textup{ESN}$ exists.
\end{corol}

We can now utilize the above corollary and show that the statement of Theorem~\ref{thm:Konvergenzaussage_verstaerkt} remains valid even if only representative sets of extreme supported efficient solutions are considered.

\begin{thm} \label{thm:Konvergenzaussage_verstaerkt-esn} 
Let the domination property \eqref{eq:domination-property} be satisfied for $P(\cU)$ and $P(\cU^\prime)$.
If $\cY$ is compact and there is a finite representative set $\Rep^\prime_\text{ESE}$ of extreme supported efficient solutions for $P(\cU^\prime)$ whose elements satisfy \eqref{eq:WCincluded}, then
\begin{enumerate}[label=(\roman*)]
\item
$x \in \Rep^\prime_\text{ESE}$ $\Rightarrow$ $x$ is extreme supported efficient for  $P(\cU)$,
\item
$x$ is extreme supported efficient for $P(\cU)$ $\Rightarrow$ $x$ is extreme supported efficient for  $P(\cU^\prime)$, and
\item
$\Rep^\prime_\text{ESE}$ is a representative set of extreme supported efficient solutions to $P(\cU)$.
\end{enumerate}
\end{thm}
\begin{proof}
\begin{enumerate}[label=(\roman*)]
\item
Let $x \in \Rep^\prime_\text{ESE}$.
Assume to the contrary that $x$ is not extreme supported efficient for $P(\cU)$, i.e., there exists a nontrivial convex combination of solutions efficient for $P(\cU)$ $x^\prime_1, \dots, x^\prime_n \in \cX$, and $\lambda \in \R_{\geq 0}$, $\sum_{i=1}^{n} \lambda_i=1$ such that
\begin{equation} \label{eq:pb-proof11}
	\sum_{i=1}^{n} \lambda_i f^\cU(x^\prime_i)
	\leq 
	f^\cU(x)
	\text,
\end{equation}
and $f^\cU(x^\prime_i) \not = f^\cU(x)$ for all $i=1,2,\dots,n$.

$\cU^\prime \subseteq \cU$, hence $f^{\cU^\prime}(x^\prime_i) \leq f^{\cU}(x^\prime_i)$, $i=1,2\dots,n$, see \eqref{eq:U'relaxation}. This leads to
\begin{align} \label{eq:pb-proof12}
	\sum_{i=1}^{n} \lambda_i f^{\cU^\prime}(x^\prime_i)
    \stackrel{\eqref{eq:U'relaxation}}{\leq}
	\sum_{i=1}^{n} \lambda_i f^{\cU}(x^\prime_i)
	\stackrel{\eqref{eq:pb-proof11}}{\leq}
	f^{\cU}(x)
	\stackrel{\eqref{eq:WCincluded}}{=}
	f^{\cU^\prime}(x)
	.
\end{align}
Hence, extreme supported efficiency of $x$ for $P(\cU^\prime)$ is contradicted
or 

\begin{equation} \label{eq:pb-proof13}
f^{\cU^\prime}(x^\prime_i)=f^{\cU^\prime}(x) \textup{ for at least one } i=1,2,\dots,n
\end{equation}
must hold.
Assume that \eqref{eq:pb-proof13} holds.
Then
\begin{equation*}
f^{\cU}(x)
\stackrel{\eqref{eq:WCincluded}}{=}
f^{\cU^\prime}(x)
\stackrel{\eqref{eq:pb-proof13}}{=}
f^{\cU^\prime}(x_i^\prime)
\stackrel{\eqref{eq:U'relaxation}}{\leq}
f^{\cU}(x_i^\prime)
\end{equation*}
follows.
Since $x_i^\prime$ is efficient for $P(\cU)$, equality holds. Hence, $f^\cU(x^\prime_i) \not = f^\cU(x)$ for all $i=1,2,\dots,n$  is contradicted.

\item
Let $x$ be  extreme supported efficient for $P(\cU)$.
Assume to the contrary that $x \in \cX$ is not extreme supported efficient for $P(\cU^\prime)$.
Then Corollary~\ref{cor:WeaklyDomByNontrivConvCombOfESN} can be applied to the problem $P(\cU^\prime)$ with $\cY=f^{\cU^\prime}(\cX)$ and
there exists a nontrivial convex combination $x^\prime_1,\dots,x^\prime_n \in \Rep^\prime_\textup{ESE}$, and $\lambda \in \R_{\geq 0}$, $\sum_{i=1}^{n} \lambda_i=1$ such that
\begin{equation} \label{eq:pb-proof14}
	\sum_{i=1}^{n} \lambda_i f^{\cU^\prime}(x^\prime_i)
	\leq 
	f^{\cU^\prime}(x)
\end{equation}
and $f^{\cU^\prime}(x^\prime_i) \not = f^{\cU^\prime}(x)$ for all $i=1,2,\dots,n$.

Note that since $x^\prime_i \in \Rep^\prime_\textup{ESE}$, $i=1,2,\dots,n$, they satisfy \eqref{eq:WCincluded}. Together with $\cU^\prime \subseteq \cU$ we receive
\begin{equation} \label{eq:pb-proof15}
	\sum_{i=1}^{n} \lambda_i f^\cU(x^\prime_i)
	\stackrel{\eqref{eq:WCincluded}}{=}
	\sum_{i=1}^{n} \lambda_i f^{\cU^\prime}(x^\prime_i)
	\stackrel{\eqref{eq:pb-proof14} }{\leq}
	f^{\cU^\prime}(x)
  	\stackrel{\eqref{eq:U'relaxation}}{\leq}
	f^{\cU}(x)
	\text.
\end{equation}
This contradicts the assumption of $x$ being extreme supported efficient for $P(\cU)$ or
\begin{equation} \label{eq:pb-proof16}
f^{\cU}(x^\prime_i)=f^{\cU}(x) \textup{ for at least one } i=1,2,\dots,n
\end{equation}
must hold.
Assume \eqref{eq:pb-proof16} holds.

Then
\begin{equation*}
f^{\cU^\prime}(x)
\stackrel{\eqref{eq:U'relaxation}}{\leq}
f^{\cU}(x)
\stackrel{\eqref{eq:pb-proof16}}{=}
f^{\cU}(x_i^\prime)
\stackrel{\eqref{eq:WCincluded}}{=}
f^{\cU^\prime}(x_i^\prime)
\end{equation*}
follows.
Since $x_i^\prime$ is efficient for $P(\cU^\prime)$, equality holds. Hence, $f^{\cU^\prime}(x^\prime_i) \not = f^{\cU^\prime}(x)$ for all $i=1,2,\dots,n$ is contradicted.

\item
Let $\Rep_\text{ESE} \subset \cX$ be a representative set of extreme supported efficient solutions for $P(\cU)$.
Analogously to the proof of Theorem~\ref{thm:Konvergenzaussage_verstaerkt}~(iii) we show that
$f^\cU(\Rep_\text{ESE}^\prime) = f^\cU(\Rep_\text{ESE})$.
\\
\underline{$\subset$:}
Let $y^\prime \in f^\cU(\Rep_\text{ESE}^\prime)$.
Then $y^\prime=f^\cU(x^\prime)$ for some $x^\prime \in \Rep_\text{ESE}^\prime$.
According to (i), $x^\prime$ is extreme supported efficient for $P(\cU)$, hence 
$y^\prime \in f^\cU(\Rep_\text{ESE})$.
\\
\underline{$\supset$:}
Let $y \in f^\cU(\Rep_\text{ESE})$.
Then $y=f^\cU(x)$ for some $x$ that is extreme supported efficient for $P(\cU)$.
According to (ii), $x$  is also extreme supported efficient for $P(\cU^\prime)$.
Hence, $x^\prime \in \Rep^\prime_\text{ESE}$ exists such that
$f^{\cU^\prime}(x) = f^{\cU^\prime}(x^\prime)$.
This leads to
\begin{equation*}
y=f^\cU(x)
\stackrel{\eqref{eq:U'relaxation}}{\geq}
f^{\cU^\prime}(x)
=
f^{\cU^\prime}(x^\prime)
\stackrel{\eqref{eq:WCincluded}}{=}
f^{\cU}(x^\prime)
\textup.
\end{equation*}
Since by assumption $y$ is extreme supported nondominated for $P(\cU)$, equality must hold true.
Thus, $y=f^{\cU}(x^\prime)$ for $x^\prime \in \Rep_\text{ESE}^\prime$ and, consequently,
$y \in f^\cU(\Rep_\text{ESE}^\prime)$.

\end{enumerate}
\end{proof}

We can now formulate the multiobjective generalization of optimization-pessimization.

\paragraph{Adaption of optimization-pessimization.}
In order to deal with the multiobjective setting algorithmically, we modify optimization-pessimization  for multiobjective problems as it is described in the following (see also Figure~\ref{fig:Opt-Pess-Scheme-Multiobjective}):
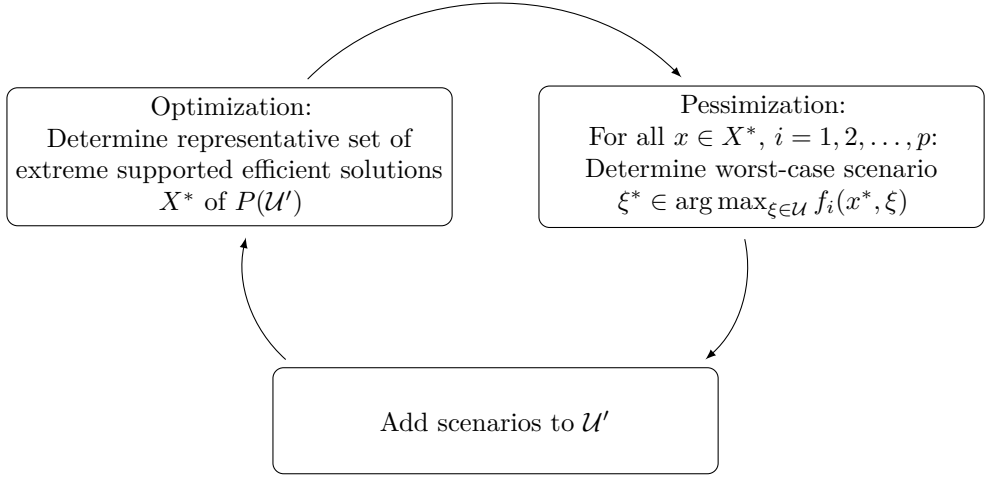
\begin{figure}[hbt!]
\small
\centering
\usetikzlibrary{shapes.geometric,arrows}
\tikzstyle{block} = [rectangle, draw, text width=16em, text badly centered, rounded corners, minimum height=4em] 
\tikzstyle{emptyblock} = []  
\begin{tikzpicture}[node distance=3.5cm, auto, arrow/.style={-latex, shorten >=1ex, shorten <=1ex, bend angle=45}]
\node [emptyblock]				(ctr) {};
\node [block, left of=ctr]		(opt)  {
	Optimization:\\ 
	Determine representative set of extreme supported efficient solutions $X^\ast$ of $P(\cU^\prime)$ 
}; 
\node [block, right of=ctr]		(pess) {
	Pessimization:\\
	For all $x \in X^\ast$, $i=1,2,\dots,p$: \\ Determine worst-case scenario
	$\xi^\ast \in \argmax_{\xi \in \cU} f_i(x^\ast,\xi)$};  
\node [block, below of=ctr, distance=1cm]		(add) {Add scenarios to $\cU^\prime$};  
\draw [arrow, bend angle=45, bend left]  (opt) to (pess);
\draw [arrow, bend angle=30, bend left]  (pess) to (add);
\draw [arrow, bend angle=30, bend left]  (add) to (opt);
\end{tikzpicture}
\caption{Optimization-pessimization for robust multiobjective optimization problems} \label{fig:Opt-Pess-Scheme-Multiobjective}
\end{figure}

When solving the \emph{optimization problem} $P(\cU^\prime)$ we do not only 
determine one  optimal solution, but a representative set $X^{\prime\ast}$ of 
extreme supported efficient solutions.
In the subsequent \emph{pessimization step} we consider \emph{all} solutions
$x \in X^{\prime\ast}$. For each of them we determine not just one worst-case scenario,
but a worst-case scenario for each of the $p$ objective functions independently.
All of these $p \cdot |X^{\prime\ast}|$
worst-case scenarios are then added to the uncertainty set.
\medskip

Algorithm~\ref{algo:OptPess-multiobjective} describes the exact procedure
and the following lemma shows its correctness.

\begin{algorithm} 
\begin{algorithmic}[hbt!]
\small
	\REQUIRE{Multi-objective robust optimization problem $P(\cU)$ as in \eqref{eq:Grundproblem-p-kriteriell}.}	\REQUIRE{Finite initial set $\cU^{(0)} \subseteq \cU$.}
	\ENSURE{Either $\cU$ finite or $\cU$ a polytope and $f_i(x,\cdot)$, $i=1,2,\dots,p$ continuous and quasi-convex.}
	\ENSURE{\eqref{eq:domination-property}, \eqref{eq:ideal-point-existent} hold for $P(\cU)$ and for $P(\cU^\prime)$ for any finite subset $\cU^\prime \subseteq \cU$.}
	\STATE{Set $k \coloneqq 0$.}
	\REPEAT
		\STATE{Set $\cU^{(k+1)} \coloneqq \cU^{(k)}$.}

    	\tikzmark{startopt44}
    	\STATE{Determine 
    	representative set for extreme supported efficient solutions $X^{(k)\ast}$ and
    	representative set for extreme supported nondominated points $Y^{(k)\ast}$ of $P(\cU^{(k)})$.
    	}
    	
    	\hfill~\tikzmark{endopt44}
    	
    	\tikzmark{startpess44}
    	\FORALL{$x^\ast \in X^{(k)\ast}$}
    	    \FORALL{$i=1,2,\dots,p$} 
    			\STATE{Determine $\xi^{\ast} \in \argmax_\cU f_i(x^\ast,\xi)$.}
    			\STATE{Add $\xi^{\ast}$ to $\cU^{(k+1)}$.}
    		\ENDFOR
    	\ENDFOR
    	
    	\hfill~\tikzmark{endpess44}
 	
		\STATE{$k \coloneqq k+1$}
	\UNTIL{$f^\cU(x^\ast)=f^{\cU^{(k-1)}}(x^\ast)$ for all $x^\ast \in X^{(k-1)\ast}$.}
	\RETURN{$X^{(k-1)\ast}$: representative set of extreme supported efficient solutions of $P(\cU)$.}
	\RETURN{$Y^{(k-1)\ast}$: set of extreme supported nondominated points of $P(\cU)$.}
	\RETURN{$\cU^\textup{FINAL}\coloneqq \cU^k$: set of worst-case scenarios.}
\end{algorithmic}
\caption{Optimization-pessimization for multi-objective robust optimization} \label{algo:OptPess-multiobjective}
\Textboxblue{startopt44}{endopt44}{Optimization}
\Textboxblue{startpess44}{endpess44}{Pessimization}
\end{algorithm}


\begin{lemma} \label{lem:OptPess-multiobjective}
Let \eqref{eq:domination-property}, \eqref{eq:ideal-point-existent} hold for $P(\cU)$ and for $P(\cU^\prime)$ for any finite subset $\cU^\prime \subseteq \cU$.
\begin{enumerate}[label=(\roman*)]
\item Let $\cU$ be finite. Then Algorithm~\ref{algo:OptPess-multiobjective} returns a representative set of extreme supported efficient solutions to $P(\cU)$ in at most $|\cU|$ iterations.
\item Let $\cU$ be a polytope or finite and $f_i(x,\cdot) \colon \conv(\cU) \to \R$, $i=1,2,\dots,p$, be continuous and quasi-convex.
Then Algorithm~\ref{algo:OptPess-multiobjective} returns a representative set of extreme supported efficient solutions to~\eqref{eq:Grundproblem-p-kriteriell} 
in at most $k$ iterations where $k$ is the number of extreme points of $\cU$,
if we choose an algorithm for the pessimization problem which always finds an extreme point of $\cU$.
\end{enumerate}
\end{lemma}

\begin{proof}
Algorithm~\ref{algo:OptPess-multiobjective} determines a representative set of extreme supported efficient solutions to $\cU^{(k-1)}$ in step $k$. 
It stops if
\begin{equation} \label{eq:helper3}
f^\cU(x^\ast)=f^{\cU^{(k-1)}}(x^\ast)
\end{equation}
for all $x^\ast \in X^{(k-1)\ast}$.

Hence, $\Rep_\textup{ESN}=X^{(k-1)\ast}$ is a representative set of extreme supported efficient solutions to $P(\cU^\prime)$ for $\cU^{(k-1)}$
whose elements satisfy 
\eqref{eq:WCincluded}.
Furthermore, $\cY=f^\cU(x)$ is compact, since it is the image of a compact set under the function $\max_{\xi \in \cU^\prime} f(x,\xi)$ that is continuous since $\cU^\prime$ is finite.
We can thus apply Theorem~\ref{thm:Konvergenzaussage_verstaerkt-esn} for $\cU^\prime=\cU^{(k-1)} \subseteq \cU$ and,
after termination, $X^{(k-1)\ast}$ is a representative set of extreme supported efficient solutions to $P(\cU)$.

We now show the bounds on the number of iterations.
\begin{enumerate}[label={ad (\roman*)}]
\item In every iteration, either at least one new worst-case scenario is added or \eqref{eq:helper3} holds and the procedure stops.
Since $\cU$ is finite, the latter happens after at most $|\cU|$ iterations.


\item Consider the pessimization problem $\textup{Pess}(x^k)$:
here we maximize a continuous function over a compact set $\cU$,
i.e., a maximum always exists.
Since $f(x,\cdot)$ is quasi-convex, the maximum is always attained at an extreme point of $\cU$. If we choose an algorithm that returns an extreme point for such optimization problems, we add a new extreme point in each iteration until \eqref{helper3} holds as in part (i).
\end{enumerate}
\end{proof}

Algorithm~\ref{algo:OptPess-multiobjective} provides a method to solve problem \eqref{eq:GrundproblemBRO} under the stated assumptions.
However, this is still challenging since in each iteration a representative set for all extreme supported efficient solutions to $P(\cU^\prime)$ for some $\cU^\prime \subset \cU$ needs to be found.
In Section~\ref{sec:MOA} we employ dichotomic search for this purpose.

\section{Algorithms for robust biobjective optimization} \label{sec:Algos}

In Sections~\ref{sec:DichotomicSearch-minmax} and~\ref{sec:OptPess-multiobjective} algorithms known from (deterministic) biobjective and (single-objective) robust optimization, respectively, have been generalized.
However, 
in each iteration of the proposed dichotomic search method (Algorithm~\ref{algo:DichotomicSearch}, Lemma~\ref{lem:DichotomicSearch-BRO-polytopeU}) a robust problem has to be solved and,
similarly,
in each iteration of the proposed optimization-pessimization method (Algorithm~\ref{algo:OptPess-multiobjective}, Lemma~\ref{lem:OptPess-multiobjective})
a multiobjective problem has to be solved.
So far, we treated these steps as if they were performed by an oracle.

In this section we put these steps into concrete terms and,  in doing so, present algorithms designed to solve uncertain biobjective problems, more specifically the problem \eqref{eq:GrundproblemBRO} as defined in Section~\ref{sec:preliminaries}.
Throughout this section we always assume that the assumptions of \eqref{eq:GrundproblemBRO}, 
i.e., (BRO-1), (BRO-2), and (BRO-3) (see page \pageref{eq:GrundproblemBRO}), hold.

Specifically, three different approaches to find minmax robust solutions for $P(\cU)$ are presented:
\begin{itemize}
\item A robust optimizer's approach (ROA): 
We view the problem \eqref{eq:GrundproblemBRO} primarily as a \emph{robust} optimization problem 
-- just with the added difficulty that it has two objective functions -- 
and, consequently, apply a method from robust optimization, namely
the generalized optimization-pessimization method (Algorithm~\ref{algo:OptPess-multiobjective}), to the problem $\textup{BRO}(\cU)$. 
The subproblem to be solved in each iteration is a
\emph{biobjective} problem $\textup{BRO}(\cU^\prime)$ with a small uncertainty
set $\cU^\prime \subseteq \cU$ which we tackle by the generalized version of
dichotomic search (Algorithm~\ref{algo:DichotomicSearch}). 
This algorithm is presented in Section~\ref{sec:ROA}.
\item A multiobjective optimizer's approach (MOA): 
We view the problem \eqref{eq:GrundproblemBRO} primarily as a \emph{biobjective} optimization problem 
-- with the added difficulty that we aim to find a \emph{robust} solution and the objective functions, thus, contain a maximum -- 
and, consequently, apply a method from biobjective optimization, namely
the generalized version of dichotomic search (Algorithm~\ref{algo:DichotomicSearch})
to the problem $\textup{BRO}(\cU)$. 
The subproblem to be solved in each iteration is a single-objective
but \emph{uncertain} problem $P(\cU, \lambda)$ which we tackle by
the optimization-pessimization method (Algorithm~\ref{algo:OptPess-single-objective}).
This algorithm is presented in Section~\ref{sec:MOA}.
\item A multiobjective optimizer's approach for bilinear problems using dualization (DA): 
As in the aforementioned approach, 
we take the multiobjective optimizer's perspective and
apply the generalized version of dichotomic search (Algorithm \ref{algo:DichotomicSearch}) to the problem $\textup{BRO}(\cU)$. 
The subproblem $P(\cU, \lambda)$ is directly solved through a reformulation in each iteration.
This algorithm is presented in Section~\ref{sec:dualize}.
\end{itemize}

Algorithms~\ref{algo:ROA}, \ref{algo:MOA}, and~\ref{algo:dichotomic_with_dual} each determine all extreme supported nondominated points and a corresponding representative set of extreme supported efficient solution for $\textup{BRO}(\cU)$.
The following lemma shows that these sets can be used to determine \emph{all} nondominated points and a representative set for all efficient solutions of $\textup{BRO}(\cU)$.

\begin{lemma} \label{lem:Construct_Paretofront_from_ESN-Points}
Let $\textup{BRO}(\cU)$ be given and let $\cX$ be a polytope.
Further, let $\cY_\textup{ESN}$, $|\cY_\textup{ESN}|<\infty$, be its set of nondominated extreme supported points and $\cX_\textup{ESE}$ a representative set of extreme supported efficient solutions.
Let $\cX_\textup{ESE}=\left\lbrace x^{(1)}, x^{(2)}, \dots, x^{(n)} \right\rbrace $,
$\cY_\textup{ESN}=\left\lbrace y^{(1)}, y^{(2)}, \dots, y^{(n)} \right\rbrace$,
$y^{(1)}_1 < y^{(2)}_1 < \dots < y^{(n)}_1$ and
$f(x^{(i)})=y^{(i)}$ for $i=1,2,\dots,n$ .
Then
\begin{align*}
\cX^\ast &\coloneqq \bigcup_{i=1,2,\dots,n-1} \left\lbrace \lambda x^{(i)} + (1-\lambda) x^{(i+1)} \colon \lambda \in (0,1) \right\rbrace
\intertext{is a representative set (of efficient solutions) and}
\cY^\ast &\coloneqq \bigcup_{i=1,2,\dots,n-1} \left\lbrace \lambda y^{(i)} + (1-\lambda) y^{(i+1)} \colon \lambda \in (0,1) \right\rbrace
\end{align*}
is the set of nondominated points of $\textup{BRO}(\cU)$.
\end{lemma}
\begin{proof}

Let $\bar{x} \in \cX^\ast$.
Then $\bar{x}=\lambda x^{(i)} + (1-\lambda) x^{(i+1)}$ for some $i=1,2,\dots,n-1$, 
$\lambda \in (0,1)$, 
and
\begin{align*}
\bar{y}
\coloneqq
f^\cU(\bar{x})
&=
\max_{\xi \in \cU} f(\lambda x^{(i)} + (1-\lambda) x^{(i+1)}, \xi) \\
&=
\max_{\xi \in \cU} \left\lbrace \lambda f(x^{(i)},\xi) + (1-\lambda) f(x^{(i+1)}, \xi) \right\rbrace \\
&\leq
\max_{\xi \in \cU} \lambda f(x^{(i)},\xi) + \max_{\xi \in \cU} (1-\lambda) f(x^{(i+1)}, \xi) \\
&= \lambda f^\cU(x^{(i)}) + (1-\lambda) f^\cU(x^{(i+1)}) \\
&= \lambda y^{(i)} + (1-\lambda) y^{(i+1)}
.
\end{align*}
However, since by 
Lemma~\ref{lem:WeaklyDomByNontrivConvCombOfESN}
we have
$\cY \subseteq \conv(\cY_\textup{ESN}) + \R_\geq^2$
and since $\{\lambda y^{(i)} + (1-\lambda) y^{(i+1)} \}$ is a facet
of $\conv(\cY_{\textup{ESN}})$,
there is no $y \in \cY$ with $y \preceq \lambda y^{(i)} + (1-\lambda) y^{(i+1)}$.
Thus, we have
$\bar{y} = \lambda y^{(i)} + (1-\lambda) y^{(i+1)}$
and $\bar{y}$ is nondominated.
This shows that the solutions in $\cX^\ast$ are efficient and the points in $\cY^\ast$ are nondominated.

It remains to be shown that all nondominated points are included in $\cY^\ast \cup \cY_{\textup{ESN}}$. This, however, follows directly from the fact that, by Lemma~\ref{lem:WeaklyDomByNontrivConvCombOfESN} $\cY \subseteq \conv(\cY_{\textup{ESN}}) + \R_geq^2$. 
\end{proof}

\subsection{A robust optimizer's approach} \label{sec:ROA}

The robust optimizer's approach is based on the idea of applying the generalization of
optimization-pessimization (Algorithm~\ref{algo:OptPess-multiobjective}). In the $k$-th iteration a representative set of extreme supported efficient solutions to $P(\cU^{(k)})$ has to be determined.
For this purpose in Algorithm~\ref{algo:ROA} we employ dichotomic search for robust biobjective linear mixed-integer optimization problems as shown possible in Section 3.2.

\begin{algorithm} 
\begin{algorithmic}[hbt!]
\small
	\REQUIRE{Biobjective mixed-integer linear robust optimization problem \eqref{eq:GrundproblemBRO}.}
	\REQUIRE{Finite initial set $\cU^{(0)} \subseteq \cU$.}
	\ENSURE{Feasible set $\cX$ is a polyhedron intersected with $\R^{n-k} \times \Z^k$ for some $k \in \{0,\ldots,n\}$.}
	\ENSURE{$\cU$ finite or $\cU$ a polytope and $f_i(x,\cdot)$, $i=1,2,\dots,p$ continuous and quasi-convex.}
	\ENSURE{\eqref{eq:domination-property}, \eqref{eq:ideal-point-existent} hold for $P(\cU)$ and for $P(\cU^\prime)$ for any finite subset $\cU^\prime \subseteq \cU$.}
	\STATE{Set $k \coloneqq 0$.}
	\REPEAT
		\STATE{Set $\cU^{(k+1)} \coloneqq \cU^{(k)}$.}
		
		\tikzmark{startopt51}
    	\STATE{Call dichotomic search (Algorithm \ref{algo:DichotomicSearch}) for $\textup{BRO}(\cU^k)$ to determine 
    	representative set for extreme supported efficient solutions $X^{(k)\ast}$ and
    	representative set for extreme supported nondominated points $Y^{(k)\ast}$.
    	}
    	
    	\hfill~\tikzmark{endopt51}
    	
    	
    	\tikzmark{startpess51}
    	\FORALL{$x^\ast \in X^{\ast}$}
			\FORALL{$i=1,2$} 
				\STATE{Determine one $\xi^{\ast} \in \argmax_\cU f_i(x^\ast,\xi)$.}
				\STATE{Add $\xi^{\ast}$ to $\cU^{(k+1)}$.}
			\ENDFOR
		\ENDFOR
		
		\hfill~\tikzmark{endpess51}
		
		\STATE{$k \coloneqq k+1$}
	\UNTIL{$f^\cU(x^\ast)=f^{\cU^{(k-1)}}(x^\ast)$ for all $x^\ast \in X^{(k-1)\ast}$.}
	\RETURN{$X^{(k-1)\ast}$: representative set of extreme supported efficient solutions of $P(\cU)$.}
	\RETURN{$Y^{(k-1)\ast}$: set of extreme supported nondominated points of $P(\cU)$.}
	\RETURN{$\cU^\textup{FINAL}\coloneqq \cU^k$: set of worst-case scenarios.}
\end{algorithmic}
\caption{Robust optimizer's approach (ROA)} \label{algo:ROA}
\Textboxblue{startopt51}{endopt51}{Optimization}
\Textboxblue{startpess51}{endpess51}{Pessimization}
\end{algorithm}

Note that Algorithm~\ref{algo:ROA} is just Algorithm~\ref{algo:OptPess-multiobjective} with the optimization step performed by dichtomic search (Algorithm~\ref{algo:DichotomicSearch}). 
Consequently, the requirements correspond to those of 
Algorithm~\ref{algo:OptPess-multiobjective} and Algorithm~\ref{algo:DichotomicSearch} 
as formulated in Lemma~\ref{lem:OptPess-multiobjective} and Lemma~\ref{lem:DichotomicSearch-BRO-polytopeU}, respectively.
This is stated in the following lemma.

\begin{lemma} \label{lem:ROA-works}
Let $\textup{BRO}(\cU)$ be given.
\begin{enumerate}[label=(\roman*)]
\item Let $\cU$ be finite. Then Algorithm~\ref{algo:ROA} returns a representative set of extreme supported efficient solutions to \eqref{eq:GrundproblemBRO} in at most $|\cU|$ iterations.
\item Let $\cU$ be a polytope or finite and $f_i(x,\cdot) \colon \conv(\cU) \to \R$, $i=1,2,\dots,p$, be continuous and quasi-convex.
Then Algorithm~\ref{algo:ROA} returns a representative set of extreme supported efficient solutions to~\eqref{eq:GrundproblemBRO}, 
in at most $k$ iterations (where $k$ is the number of extreme points of $\cU$)
if we choose an algorithm for the pessimization problem which always finds an extreme point of $\cU$.
\end{enumerate}
\end{lemma}

\begin{proof}
By Corollary~\ref{cor:PropertiesForBRO}, $\textup{BRO}(\cU)$ satisfies \eqref{eq:domination-property} and \eqref{eq:ideal-point-existent}.
Algorithm~\ref{algo:ROA} is the same as Algorithm~\ref{algo:OptPess-multiobjective}, but for $p=2$ and with dichotomic search (Algorithm~\ref{algo:DichotomicSearch}) specified in the optimization step.
Lemma~\ref{lem:DichotomicSearch-BRO-polytopeU} justifies that 
dichotomic search works correctly for \ref{eq:GrundproblemBRO}.
Consequently, we may use dichotomic search in line~4 of Algorithm~\ref{algo:OptPess-multiobjective}.
Under \eqref{eq:domination-property} and \eqref{eq:ideal-point-existent} for $\textup{BRO}(\cU)$ and $\textup{BRO}(\cU^\prime)$ for all finite sets $\cU^\prime \subseteq \cU$
Lemma~\ref{lem:OptPess-multiobjective} gives us correctness of Algorithm~\ref{algo:OptPess-multiobjective} and hence also of Algorithm~\ref{algo:ROA}.
%
%
\end{proof}

Note that if Algorithm 5.1 is stopped before the stopping criterion in line 12 is met, the set $\{ f^\cU(k-1)(x): x\in \cX^{(k-1)\ast} \}$  and  $\{ f^\cU(x): x\in X^{(k-1)\ast} \}$ provide lower and upper bounds with respect to the upper setless order, as we have shown in Lemma~\ref{lem:upper-setless-bounds}.
Using convex combinations of subsequent points in these sets like we did in Lemma~\ref{lem:Construct_Paretofront_from_ESN-Points} for $Y^\ast$, we obtain bounds on the region in which the Pareto frontier $Y^\ast$ will lie.
In this sense, Algorithm~\ref{algo:ROA} can be used as an approximation algorithm for \eqref{eq:GrundproblemBRO}.




\subsection{A multiobjective optimizer's approach} \label{sec:MOA}

The multiobjective optimizer's approach is based on the idea of applying dichotomic search (Algorithm~\ref{algo:DichotomicSearch}) as introduced in Section~\ref{sec:DichotomicSearch-intro} directly to $P(\cU)$.
In each iteration of dichotomic search, we have to solve the scalarized weighted-sum problem
\begin{align} \label{eq:PUlambda}
P(\cU,\lambda)
&&
\min_{x \in \cX} \lambda_1 f_1^\cU(x) + \lambda_2 f_2^\cU(x) + \dots + \lambda_p f_p^\cU(x)
&&
\end{align}
for $p=2$ and given weights $\lambda\in \R_{\succeq 0}$.
In order to do this, we utilize optimization-pessimization for single-objective robust optimization as reviewed in Section~\ref{sec:OptPess-intro}: We solve a sequence of problems $P(\cU^0,\lambda),P(\cU^1,\lambda),\dots,P(\cU^k,\lambda)$ until it is guaranteed that $P(\cU^k,\lambda)$ and $P(\cU,\lambda)$ share a representative set of extreme supported minmax robust efficient solutions.
As in Section~\ref{sec:OptPess} we exploit the fact, that for finite sets $\cU^\prime$ a problem $P(\cU^\prime,\lambda)$ is easier to solve than $P(\cU,\lambda)$ as it can be written as a problem with finitely many constraints.
For solving the scalarization we assumed an oracle in Algorithm~\ref{algo:DichotomicSearch}.
Now we want to be more specific.
We first reformulate problem \eqref{eq:PUlambda} such that we can apply optimization-pessimization
 (see Section~\ref{sec:OptPess-intro}) for its solution. 
 This is done in the next lemma.

\begin{lemma} \label{lem:Hilfslemma}
Let $\lambda \in \R^2_{\succeq 0}$ be fixed. Then $P(\cU,\lambda)$ can be transformed to
\begin{align} \label{eq:PU_bar}
\bar{P}(\cU,\lambda)
&&
\min_{x \in \cX} \sup_{\bar{\xi} \in \bar{\cU}} \bar{f}_\lambda(x,\bar{\xi})
\textup{,}
&&
\end{align}
i.e., a problem of type $\textup{P}^\textup{single}$ as introduced in \eqref{eq:MinSupProblem},
for $\bar{\cU} \coloneqq \bigtimes_{i=1,2,\dots,p} \cU$, 
$\bar{\xi} \coloneqq (\xi_1,\xi_2,\dots,\xi_p)$ and 
$\bar{f}_\lambda(x,\bar{\xi}) \coloneqq \sum_{i=1}^p \lambda_i f_i(x,\xi_i)$.
\end{lemma}
\begin{proof}
We reformulate \eqref{eq:PUlambda} as follows:
\begin{align*}
\min_{x \in \cX} 
	\left\lbrace 
		\lambda_1 f_1^\cU(x) + \lambda_2 f_2^\cU(x) + \dots + \lambda_p f_p^\cU(x) 
	\right\rbrace
&=
\min_{x \in \cX} 
\left\lbrace \sum_{i=1}^p \lambda_i \sup_{\xi \in \cU} f_i(x,\xi) \right\rbrace\\
&=
\min_{x \in \cX} 
\sup_{(\xi_1,\xi_2,\dots,\xi_p) \in \cU^p} 
\left\lbrace
\sum_{i=1}^p \lambda_i  f_i(x,\xi_i)
\right\rbrace
\\
&=
\min_{x \in \cX} \sup_{\bar{\xi} \in \bar{\cU}} \bar{f}_\lambda(x,\bar{\xi})
\text.
\end{align*}
\end{proof}

Lemma~\ref{lem:Hilfslemma} shows that $P(\cU,\lambda)$ can be solved by solving a single-objective robust optimization problem $\bar{P}(\bar{\cU},\lambda)$, i.e., of type $P^\textup{single}$, as has been introduced in \eqref{eq:MinSupProblem}.

Algorithm~\ref{algo:MOA} describes a \emph{basic version} of the multiobjective
optimizer's approach. 
Its correctness is shown in the following lemma.

\begin{algorithm} 
\begin{algorithmic}[!hbt]
\small
	\REQUIRE{Biobjective mixed-integer linear robust optimization problem \eqref{eq:GrundproblemBRO}.}
	\REQUIRE{Finite initial set $\cU^{(0)} \subseteq \cU$.}
	\ENSURE{Feasible set $\cX$ is a polyhedron intersected with $\R^{n-k} \times \Z^k$ for some $k \in \{0,\ldots,n\}$.}
	\ENSURE{$\cU$ finite or $\cU$ a polytope and $f_i(x,\cdot)$, $i=1,2,\dots,p$ continuous and quasi-convex.}
	\ENSURE{\eqref{eq:domination-property}, \eqref{eq:ideal-point-existent} hold for $P(\cU)$ and for $P(\cU^\prime)$ for any finite subset $\cU^\prime \subseteq \cU$.}

    \STATE{Initialize $\mathcal{L} \coloneqq \emptyset$ }     \COMMENT{$\mathcal{L}$ will contain list of tuple images $(y^l,y^r)$ satisfying $y^l_1 < y^r_1, y^l_2 > y^r_2$}
	
	\tikzmark{MOA-start1}
	\STATE{Call optimization-pessimization (Algorithm \ref{algo:OptPess-single-objective}) 
	on $\min_{x \in \cX} f^\cU_1(x)$
	 with initial set $\cU^{(0)}$ to determine $\epsilon_1$, $\cU^\text{FINAL}$, and $\xi^\textup{WC}$.}
	\STATE{Call optimization-pessimization (Algorithm \ref{algo:OptPess-single-objective}) 
	on $\min_{x \in \cX} \{ f^\cU_2(x) \colon \max_{\xi \in \cU} f_1(x,\xi) \leq \epsilon_1 \}$
	with initial set $\cU^\text{FINAL}$ to determine optimal solution $x^L$.}
    \STATE{Set $y^L \coloneqq f^\cU(x^L) $.}
	\STATE{Call optimization-pessimization (Algorithm \ref{algo:OptPess-single-objective}) 
	on $\min_{x \in \cX} f^\cU_2(x)$
	 with initial set $\cU^{(0)}$ to determine $\epsilon_2$, $\cU^\text{FINAL}$, and $\xi^\textup{WC}$.}
	\STATE{Call optimization-pessimization (Algorithm \ref{algo:OptPess-single-objective}) 
	on $\min_{x \in \cX} \{ f^\cU_1(x) \colon \max_{\xi \in \cU} f_2(x,\xi) \leq \epsilon_2 \}$
	with initial set $\cU^\text{FINAL}$ to determine optimal solution $x^R$.}.
    \STATE{Set $y^R \coloneqq f^\cU(x^R)$.}
    
    \hfill~\tikzmark{MOA-end1}

	\IF{$y^L=y^R$}
		\STATE{STOP. Only one nondominated image found.}
		\RETURN{$Y^\ast = \{y^L\}, X^\ast = \{x^L\}$.}
	\ELSE
		\STATE{$Y^\ast = \{y^L,y^R\}, X^\ast = \{x^L,x^R\}, \mathcal{L}=\{(y^L,y^R)\}$.}
	\ENDIF
	\WHILE{$L \not = \emptyset $}
		\STATE{Remove element $(y^l,y^r)$ from $\mathcal{L}$.}
		\STATE{Compute $\lambda \coloneqq (y^l_2-y^r_2,y^r_1-y^l_1)$.}

		
		
		\tikzmark{MOA-start2b}
		\STATE{Call optimization-pessimization (Algorithm~\ref{algo:OptPess-single-objective})
		on $\min_{x \in \cX} \bar{f_\lambda}(x)$
		with initial set $\cU^{(0)}$ to determine optimal solution $x^\ast$.}
		\STATE{Set $y^\ast \coloneqq \bar{f_\lambda}(x^\ast)$.}
		
		\hfill~\tikzmark{MOA-end2b}

		\IF{$\lambda^T y^\ast \not = \lambda^T y^l$}
			\STATE{Add $y^\ast$ to $Y^\ast$, add $x^\ast$ to $X^\ast$.}
			\STATE{Add $(y^l,y^\ast),(y^\ast,y^r)$ to $\mathcal{L}$}
		\ENDIF
	\ENDWHILE
	\RETURN{$X^{\ast}$: representative set of extreme supported efficient solutions of $P(\cU)$.}
	\RETURN{$Y^{\ast}$: set of extreme supported nondominated points of $P(\cU)$.}
\end{algorithmic}
\caption{Multiobjective optimizer's approach (MOA)} \label{algo:MOA}
\Textboxred{MOA-start1}{MOA-end1}{Determine lexicographic solutions}
\Textboxred{MOA-start2b}{MOA-end2b}{Solve weighted-sum problem $\bar{P}(\cU,\lambda)$}
\end{algorithm}

\begin{lemma} \label{lem:MOA-works}
Let $\textup{BRO}(\cU)$ be given.
Then Algorithm~\ref{algo:MOA} returns a representative set of extreme supported efficient solutions to \eqref{eq:GrundproblemBRO} after a finite number of iterations.
\end{lemma}

\begin{proof}
Algorithm~\ref{algo:MOA} is dichotomic search (Algorithm~\ref{algo:DichotomicSearch}),
where we specified the algorithm for steps 2-3, 5-6, 17-18,
namely by solving $\textup{BRO}(\cU,\lambda)$ by optimization-pessimization  (Algorithm~\ref{algo:OptPess-single-objective}) in each iteration.
Since $\textup{BRO}(\cU)$ meets the requirements of 
Lemma~\ref{lem:DichotomicSearch-BRO-finiteU} (in case $\cU$ is finite)
or
Lemma~\ref{lem:DichotomicSearch-BRO-polytopeU} (in case $\cU$ is a polytope),
Algorithm~\ref{algo:DichotomicSearch} returns a representative set of extreme supported efficient solutions and a set of extreme supported nondominated solutions after finitely many iterations.

It remains to show that lines 2-3, 5-6 and 17-18 in Algorithm~\ref{algo:MOA}
are correct specifications of the same lines of Algorithm~\ref{algo:DichotomicSearch}.

For lines 2 and 5 this is straightforward as the problems 
\begin{equation} \label{eq:hilfs}
\min_{x\in\cX} f_i^\cU(x),
\end{equation}
$i=1,2$, are single-objective robust optimization problems.
Since $\cU$ is a polytope or finite and $f_i(x,\cdot) \colon \cU \to \R$, $i=1,2$, are continuous and quasi-convex, Lemma~\ref{lem:OptPessSingleObjective} can be applied and optimization-pessimization  (Algorithm~\ref{algo:OptPess-single-objective})
solves \eqref{eq:hilfs}.

The problems in lines 3 and 6 are also of type \eqref{eq:hilfs} only with one additional constraint, i.e., with feasible set is
\[
\cX^\prime_j \coloneqq \{x \in \cX: \max_{\xi \in \cU} f_j(x,\xi) \leq \epsilon_j \}, j=2,1.
\]


In lines~17-18 of Algorithm~\ref{algo:DichotomicSearch} the problem $P(\cU,\lambda)$
is to be solved for some $\lambda \in \R_{\succeq 0}$. 
By Lemma~\ref{lem:Hilfslemma} this can be done by solving $\bar{P}(\cU,\lambda)$ instead which is done in lines 17-18 of Algorithm~\ref{algo:MOA}. 
Since continuity and quasi-convexity of $\bar{f}$ are inherited from continuity and quasi-convexity of $f_1$ and $f_2$, Lemma~\ref{lem:OptPess-multiobjective} can be applied and optimization-pessimization returns a robust solution to $P(\cU,\lambda)$.
\end{proof}

\paragraph*{Warm start modifications.}

In the basic version of Algorithm \ref{algo:MOA} the cutting plane method is initialized with $\cU^\text{(0)}$ in lines 2,5 and 17. 
A possible modification of Algorithm \ref{algo:MOA} is to start the cutting plane method with a larger set $\cU^\prime$ that includes some additional scenarios that have been generated in previous iterations but that is still guaranteed to be finite.
This way, previously generated cutting planes are not forgotten.
Specifically, we propose two modifications:
\begin{itemize}
\item Variant 1 (MOA-ws1): We initialize optimization-pessimization with all previously generated scenarios. To this end, we modify lines~5 and~17 such that the cutting plane method is initialized with $\cU^\text{FINAL}$. This way, $\cU^\text{FINAL}$ grows monotonically.
\item Variant 2  (MOA-ws2): We initialize the cutting plane method with those scenarios that turned out to be worst-case scenarios for a previously found solution optimal $x$. After lines 2-3, 5-6, and 17-18 the worst-case scenarios $\xi^{\text{WC}}$ for $x^L$, $x^R$, and $x^\ast$, respectively, are added to $\cU^\text{(0)}$ and the set grows monotonically, but is much smaller than the set in Variant 1.
\end{itemize}

As Lemmas~\ref{lem:Hilfslemma} and~\ref{lem:MOA-works} above only assume finiteness of the initial uncertainty set their validity is not affected by these modifications.

\subsection{A multiobjective optimizer's approach for bilinear problems} \label{sec:dualize}

In this section, we confine ourselves to a special class of problems:
biobjective mixed-integer linear robust optimization
problems \eqref{eq:GrundproblemBRO} which satisfy not only (BRO-1), (BRO-2), and (BRO-3) as before,
but also the following additional properties:
\begin{itemize}
\item the uncertainty set $\cU$ is as a polytope
$\cU = \{\xi \in \R^m \colon C \xi \leq d \}$
for a matrix $C \in \R^{m^\prime \times m}$ 
and a vector $d \in \R^{m^\prime}$,
and
\item the functions $f_1,f_2 \colon \cX \times \cU \to \R$ are not only linear in $x$ for every fixed $\xi \in \cU$ as required in (BRO-3),
but also linear in $\xi$ for each $x$,
i.e., they are \emph{bilinear} functions.
\end{itemize}

The following lemma shows that under these assumptions a biobjective mixed-integer linear \emph{minmax} optimization problem can be reformulated as a biobjective mixed-integer linear \emph{minimization} problem.

\begin{lemma}\label{lem:Duality2}
We consider the uncertain problem
\begin{align*} \tag{\ref{eq:Grundproblem-p-kriteriell} revisited}
P(\cU)
&&
\min_{x \in \cX}  f^\cU(x)
\textup{.}
&&
\end{align*}
Let the uncertainty set be a non-empty polytope 
$\cU = \{\xi \in \R^m \colon C \xi \leq d \}$,
with $C \in \R^{m^\prime \times m}$, $d \in \R^{m^\prime}$,
and
let the functions $f_i(x,\xi)$, $i=1,2,\dots,p$, be linear in $\xi$ for each $x$, i.e.,
  \[ f_i(x,\xi) \coloneqq \left[ \hat{c_i}(x) \right]^t \xi \]
for functions $\hat{c_i} \colon \cX \to \R^m$,
$i=1,2,\dots,p$.

Let $\lambda \in \R_\succeq^p$.
Then a solution $x^\ast \in \cX$ is optimal for the scalarized problem
\begin{align*} \tag{\ref{eq:PUlambda} revisited}
P(\cU,\lambda)
&&
\min_{x \in \cX} \lambda^t f^\cU(x)
&&
\end{align*}
if and only if there exist $\pi^{(1)\ast}, \dots, \pi^{(p)\ast}  \in \R^{m^\prime}$ such that
$(x^\ast,\pi^{(1)\ast},\dots,\pi^{(p)\ast})$ is optimal for 
\begin{align*}
D(\cU,\lambda) 
&&
\min_{x \in \cX, \pi^{(1)}, \dots, \pi^{(p)}  \in \R^{m^\prime}} \left\lbrace d^t \sum_{i=1}^p \lambda_i \pi^{(i)} \colon C^t \pi^{(i)}= \hat{c_i}(x), \pi^{(i)} \geq 0, i=1,2,\dots,p \right\rbrace
&&
\end{align*}
More precisely, let $x$ be fixed and let $(\pi^{(1)\ast},\dots,\pi^{(p)\ast})$ be an optimal solution to
\begin{equation*}
\min_{\pi^{(1)}, \dots, \pi^{(p)}  \in \R^{m^\prime}} \left\lbrace d^t \sum_{i=1}^p \lambda_i \pi^{(i)} \colon C^t \pi^{(i)}= \hat{c_i}(x), \pi^{(i)} \geq 0, i=1,2,\dots,p \right\rbrace
\end{equation*}
with optimal objective function value $z$. Then $z=\lambda^t f^\cU(x)$ and 
for all $i=1,2,\dots,p$ with $\lambda_i > 0$
\begin{align} \label{eq:DualityIdentity-VAR2+}
d^t \pi^{(i)\ast}= \max_{\xi \in \cU} \left[\hat{c_i}(x)\right]^t \xi
\textup{.}
\end{align}
\end{lemma}

\begin{proof}
 First note that $D(\cU,\lambda)$ is equivalent to 
\begin{equation*}
\hspace{1cm} \min_{x \in \cX} \min_{\pi^{(1)}, \dots, \pi^{(p)}  \in \R^{m^\prime}}  \left\lbrace d^t \sum_{i=1}^p \lambda_i \pi^{(i)} \colon C^t \pi^{(i)}= \hat{c_i}(x), \pi^{(i)} \geq 0, i=1,2,\dots,p \right\rbrace
\end{equation*}
which can be interpreted as optimization problem
\begin{align*}
\bar{D}(\cU,\lambda)
&&
\min_{x \in \cX} g_\lambda(x) 
&&
\end{align*}
with 
\begin{equation*}
g_\text{$\lambda$}(x) \coloneqq \min_{\pi^{(1)}, \dots, \pi^{(p)}  \in \R^{m^\prime}}  \left\lbrace d^t \sum_{i=1}^p \lambda_i \pi^{(i)} \colon C^t \pi^{(i)}= \hat{c_i}(x), \pi^{(i)} \geq 0, i=1,2,\dots,p \right\rbrace 
\textup.
\end{equation*}
We now need to show that  the objective function and the feasible set of $P(\cU,\lambda)$ and $\bar{D}(\cU,\lambda)$ coincide.
Specifically, we show
\begin{equation*}
\lambda^t f^\cU(x)=g_\lambda(x)
\end{equation*}
for all $x \in \cX$.

We first note that $\cU$ is a compact set, 
hence for any fixed $x \in \cX$ and  any $i=1,2,\dots,p$ the linear program
\begin{equation*}
\max \left\lbrace \left[ \hat{c_i}(x) \right]^t \xi \colon \xi \in \cU \right\rbrace
\end{equation*}
has an optimal solution.
Using that $\cU = \{\xi \in \R^m \colon C \xi \leq d \}$ we 
hence get from linear programming duality for $i=1,2,\dots,p$
and fixed $x \in \cX$  that
\begin{align} \label{eq:DualityMinIsMax+}
\max \left\lbrace \left[\hat{c_i}(x)\right]^t \xi \colon C \xi \leq d, \xi \in \R^m  \right\rbrace
= \min \left\lbrace d^t \pi^{(i)} \colon C^t \pi^{(i)} = \hat{c_i}(x), \pi^{(i)} \geq 0, \pi^{(i)} \in \R^{m^\prime}  \right\rbrace,
\end{align}
i.e., for any fixed $x \in \cX$ and $i=1,\ldots,p$, 
an optimal solution $\pi^{(i) \ast}$ to the right hand side satisfies
\begin{equation*}
d^t \pi^{(i) \ast}= \max_{\xi \in \cU} \left[\hat{c_i}(x)\right]^t \xi
\end{equation*}
which shows \eqref{eq:DualityIdentity-VAR2+}.
We can now derive
\begin{align*}
\lambda^t f^\cU(x)
&= \sum_{i=1}^{p}
\lambda_i \underbrace{\max_{\xi \in \R^m} \left\lbrace \left[\hat{c_i}(x)\right]^t \xi \colon C \xi \leq d \right\rbrace}_{=f_i^\cU(x)}\\
&\stackrel{\eqref{eq:DualityMinIsMax+}}{=} \sum_{i=1}^{p}
\lambda_i \min_{\pi^{(i)} \in \R^{m^\prime}} \left\lbrace d^t \pi^{(i)} \colon C^t \pi = \hat{c_i}(x), \pi^{(i)} \geq 0 \right\rbrace\\
&= \min_{\pi^{(1)}, \dots, \pi^{(p)}  \in \R^{m^\prime}} \left\lbrace d^t \sum_{i=1}^p \lambda_i \pi^{(i)}  \colon  C^t \pi^{(i)}= \hat{c_i}(x), \pi^{(i)} \geq 0 \right\rbrace \text,
\end{align*}
where the last step puts the single optimization problems together into a bigger (still separable) problem.
Thus, for any fixed $x$ the objective values of $P(\cU,\lambda)$ and $\bar{D}(\cU,\lambda)$ coincide and hence $x$ is optimal to $P(\cU,\lambda)$ if
and only if it is optimal to $\bar{D}(\cU,\lambda)$. 
\end{proof}

As in Section~\ref{sec:MOA}, we apply dichotomic search to $P(\cU)$ and solve $P(\cU,\lambda)$ for different weights $\lambda \in \R^p_\succeq$.
However, unlike in Section~\ref{sec:MOA} we do not solve $P(\cU,\lambda)$ with an iterative approach, but adopt the other approach described by \cite{HertogEtAl2015}: reformulation of $P(\cU,\lambda)$.
More specifically, we weaponize Lemma~\ref{lem:Duality2} and choose to solve 
\begin{align*}
D(\cU,\lambda) 
&&
z^\ast(\cU,\lambda)
\coloneqq
\min_{x \in \cX, \pi^{(1)}, \dots, \pi^{(p)}  \in \R^{m^\prime}} \left\lbrace d^t \sum_{i=1}^p \lambda_i \pi^{(i)} \colon C^t \pi^{(i)}= \hat{c_i}(x), \pi^{(i)} \geq 0, i=1,2,\dots,p \right\rbrace
&&
\end{align*}
instead of $P(\cU,\lambda)$.

This leads to Algorithm \ref{algo:dichotomic_with_dual}.
The following lemma shows correctness.

\begin{algorithm} 
\begin{algorithmic}[hbt!]
\small
	\REQUIRE{Biobjective mixed-integer linear robust optimization problem \eqref{eq:GrundproblemBRO}.}
	\REQUIRE{Finite initial set $\cU^{(0)} \subseteq \cU$.}
	\ENSURE{Feasible set $\cX$ is a polyhedron intersected with $\R^{n-k} \times \Z^k$ for some $k \in \{0,\ldots,n\}$.}
	\ENSURE{$\cU$ a polytope and $f_i(x,\cdot)$, $i=1,2,\dots,p$ continuous and quasi-convex.}
	\ENSURE{\eqref{eq:domination-property}, \eqref{eq:ideal-point-existent} hold for $P(\cU)$ and for $P(\cU^\prime)$ for any finite subset $\cU^\prime \subseteq \cU$.}
	\ENSURE{$f(x,\cdot) \colon \cU \to \R^p$ linear}
	\STATE{Initialize $\mathcal{L} \coloneqq \emptyset$ }     \COMMENT{$\mathcal{L}$ will contain list of tuple images $(y^l,y^r)$ satisfying $y^l_1 < y^r_1, y^l_2 > y^r_2$}
	
	\tikzmark{start-dual-1}
	\STATE{Determine optimal objective value $\epsilon_1$ of $D(\cU,(1,0))$}
	\STATE{Determine $x^L \in \argmin_\cX \{ g_{(0,1)}(x) \colon g_{(1,0)}(x) \leq \epsilon_1 \} $}
	\STATE{Set $y^L \coloneqq (\epsilon_1,g_{(0,1)}(x^L))^t$ } 
	\STATE{Determine optimal objective value $\epsilon_2$ of $D(\cU,(0,1))$}
	\STATE{Determine $x^R \in \argmin_\cX \{ g_{(1,0)}(x) \colon g_{(0,1)}(x) \leq \epsilon_2 \} $}
	\STATE{Set $y^R \coloneqq (g_{(1,0)}(x^L),\epsilon_2)^t$ } 
	
	\hfill~\tikzmark{end-dual-1}
	
	\IF{$y^L=y^R$}
		\STATE{STOP. Only one nondominated image found}
		\RETURN{$Y^\ast = \{y^L\}, X^\ast = \{x^L\}$}
	\ELSE
		\STATE{$Y^\ast = \{y^L,y^R\}, X^\ast = \{x^L,x^R\}, \mathcal{L}=\{(y^L,y^R)\}$}
	\ENDIF
	\WHILE{$L \not = \emptyset $}
		\STATE{Remove element $(y^l,y^r)$ from $\mathcal{L}$}
		\STATE{Compute $\lambda \coloneqq (y^l_2-y^r_2,y^r_1-y^l_1)$.}
		
		\tikzmark{start3}
		\STATE{Find one optimal solution $(x^\ast, \pi^{(1)},\dots,\pi^{(k)})$ for $D(\cU,\lambda) $.}
		\STATE{Set $y_i^{\ast} = d^t \pi^{(i)\ast}$ for $i=1,2$.}
		
		\hfill~\tikzmark{end3}
		
		\IF{$\lambda^T y^\ast \not = \lambda^T y^l$}
			\STATE{Add $y^\ast$ to $Y^\ast$, add $x^\ast$ to $X^\ast$.}
			\STATE{Add $(y^l,y^\ast),(y^\ast,y^r)$ to $\mathcal{L}$}
		\ENDIF
	\ENDWHILE
	\RETURN{$X^{\ast}$: representative set of extreme supported efficient solutions of $P(\cU)$.}
	\RETURN{$Y^{\ast}$: set of extreme supported nondominated points of $P(\cU)$.}
	\Textboxblue{start-dual-1}{end-dual-1}{Determine lexicographic solutions}
	\Textboxblue{start3}{end3}{Solve $D(\cU,\lambda)$}
\end{algorithmic} 
\caption{Multiobjective optimizer's approach with dualization (DA)} \label{algo:dichotomic_with_dual}
\end{algorithm}

\begin{lemma} \label{lem:CorrectnessDualityAlgorithm}
Let $\textup{BRO}(\cU)$ with 
an nonempty polytope explicitly stated as $\cU = \{\xi \in \R^m \colon C \xi \leq d \}$
for a matrix $C \in \R^{m^\prime \times m}$ and a vector $d \in \R^{m^\prime}$ as uncertainty set
and bilinear functions  $f_1, f_2 \colon \cX \times \cU \to \R$ be given.
Then Algorithm~\ref{algo:dichotomic_with_dual} solves~\eqref{eq:GrundproblemBRO}.
\end{lemma}
\begin{proof}
The assumptions of Lemma~\ref{lem:DichotomicSearch-BRO-polytopeU} are satisfied since (BRO-1) and (BRO-3) hold and $\cU$ is a polytope.
Hence dichotomic search can be applied to $\min_{x\in\cX} f^\cU(x)$.
It remains to be shown that $P(\cU,\lambda)$  is solved correctly throughout the algorithm.
Lemma~\ref{lem:Duality2} shows that robust solutions of $P(\cU,\lambda)$ can be determined by solving $D(\cU,\lambda)$ (lines 2, 5 17) and the corresponding point on the Pareto front can be computed by $y_i^\ast = d^t \pi^{(i)\ast}$ (see line 16).
\end{proof}

\section{Numerical results} \label{sec:NumericalResults}
We implemented Algorithms  \ref{algo:ROA}, \ref{algo:MOA} and \ref{algo:dichotomic_with_dual} and conducted computational experiments.

\paragraph{Structure of the problems.}
We restricted ourselves to a certain class of biobjective optimization problems:
The objective functions $f_i \colon \cX \times \cU \to \R^2$, $i=1,2$, were assumed to be bilinear, and the feasible set and uncertainty set were polytopes or discrete sets.
More specifically, we considered problems
\begin{align*}
P(\cU)&& 
\left\{ 
\min_{x\in\cX} \begin{pmatrix}
\max_{\xi \in \cU} \xi^t M_1 x \\
\max_{\xi \in \cU} \xi^t M_2 x
\end{pmatrix}
\right\}_{\xi \in \cU}
&&
\end{align*}
with
\begin{align*}
\cX &= \{ x \in \R^n \colon L^x \leq x_i \leq U^x, Ax \leq b\}
\textup{ or }&
\cX &= \{ x \in \Z^n \colon L^x \leq x_i \leq U^x, Ax \leq b\},\\
\cU &= \{ \xi \in \R^m \colon L^\xi \leq \xi_i \leq U^\xi, C \xi \leq d\} 
\textup{ or }& 
\cU &= \{ \xi \in \Z^m \colon  L^\xi \leq \xi_i \leq U^\xi, C \xi \leq d\}. \\
\end{align*}
The lower and upper bounds $L^x,U^x,L^\xi,U^\xi$ are added to ensure that $\cX$ and $\cU$ are subsets of the boxes $[L^x,U^x]^n$ and $[L^\xi,U^\xi]^m$, respectively, and, thus, are bounded as it is required. 
We chose $L^x=1, U^x=200, L^\xi=-100$ and $U^\xi=100$. 
By doing so we avoid problems where $0_n \in \cX$ and $0_m \in \interior(\cU)$, since this would imply that $x=0$ is a trivial minimizer of $f_i^\cU(x)=\max_{\xi \in \cU} \xi M_i x$, $i=1,2$.


 

\paragraph{Generating instances}
We created 100 instances of $\textup{BRO}(\cU)$
with $A \in \Z^{30 \times 5}$ and $C \in \Z^{30 \times 5}$.
To obtain instances with smaller number of constraints, as used in our experiments, we removed constraints from these initial instances.  
This makes it easier to draw conclusions when comparing algorithm performance for different values of $n^\prime$ and $m^\prime$.
The entries of the matrices $A \in \Z^{n^\prime \times n}$ and $C\in \Z^{m^\prime \times m}$ as well as the entries of $M_1,M_2 \in \Z^{m \times n}$ determining the objective function are randomly and independently generated uniformly distributed integers in $\{-100,-99,\dots, 99,100\}$.

Equally, $\tilde{b}_i$, $i=1,2,\dots,n^\prime$ and $\tilde{d}_j$, $j=1,2,\dots,m^\prime$ are randomly generated uniformly distributed integers in $\{50,51,\dots, 99,100\}$.
We then set $\bar{x} \coloneqq (100,100,\dots,100)^t \in \Z^n$ and $\bar{\xi} \coloneqq (0,0,\dots,0)^t \in \Z^m$.
Let $A_i$, $i=1,2,\dots,n^\prime$ and $C_j$, $j=1,2,\dots,m^\prime$ denote the the columns of $A$ and $C$.
By setting the right hand-side coefficients
$b_i \coloneqq A_i^t x_0 + \tilde{b}_i \left\Vert A_i \right\Vert_2$ for $i=1,2,\dots,n^\prime$
and
$d_j \coloneqq C_j^t \xi_0 + \tilde{d}_j \left\Vert C_j \right\Vert_2$ for $j=1,2,\dots,m^\prime$,
we guarantee that the spheres
$\{x \in \R^n \colon \left\Vert x-\bar{x} \right\Vert_2 \leq 50 \}$
and
$\{\xi \in \R^m \colon \left\Vert \xi-\bar{\xi} \right\Vert_2 \leq 50 \}$
are included in $\cX$ and $\cU$, respectively.
See \cite{chandrasekaran2014integer} for more on this.

\paragraph{Implementation}
We used C++ to implement our algorithms. Whenever a linear or integer optimization problem has to be solved, Gurobi 2.3 is called (with default settings). We use Gurobi's capacity to provide solution that are known to be basic solutions. The implementations were tested on a computer with 16 GB RAM, AMD Ryzen 5 PRO 2500U, 2.00 GHz.

\subsection{Evaluation of the algorithms}
In this section we evaluate the performance of the algorithms for instances of different types (polytopal and discrete sets $\cX$ and $\cU$) and different sizes by varying the number of considered constraints $n^\prime$ and $m^\prime$, respectively.

\paragraph{Discrete feasible set and discrete uncertainty set}
First, let us consider problems with a discrete feasible set and a discrete uncertainty set.
For such instances, 
the robust optimizer's approach (ROA, Algorithm~\ref{algo:ROA}) and
the multiobjective optimizer's approach (MOA, Algorithm~\ref{algo:MOA}) in its baseline version and with its two warm-start modifications are available.
The dualization approach (DA, Algorithm~\ref{algo:dichotomic_with_dual}) cannot solve such instances as it requires a polytope as uncertainty set.

Figure \ref{fig:ddX} shows the average running time of our algorithms. Each data point is the average over 100 instances with $n^\prime$ constraints on the feasible set.
The number of variables for the feasible set $n$, the number of variables for the uncertainty set $m$ and the number of constraints for the uncertainty set $m^\prime$ are all fixed and set at 5.

\begin{figure}[hbt!]
\centering
\includegraphics[width=0.8\textwidth]{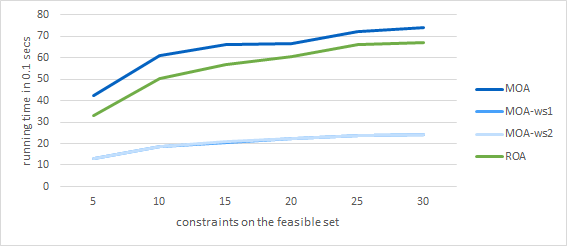}
\caption{Average running time of our four algorithms for 100 instances as a function of $n^\prime$ with $n=m=m^\prime=5$ fixed, $\cX$ and $\cU$ discrete (the lines for MOA-ws1 and MOA-ws2 overlap and are hard to see)} \label{fig:ddX}
\end{figure}

Independently of $n^\prime$, the robust optimizer's approach -- where the uncertainty set $\cU^{(k)}$ increases monotonously -- is faster than the baseline version of the multiobjective optimizer's approach.
However, the warm start modifications to the latter method turn out to be significant improvements over the baseline version: with those the multiobjective optimizer's approach performs faster.
We see a clear increase in running time when going from 5 to 10 constraints for all tested methods, but above that point an increasing number number of constraints does not seem to make the problem much harder to solve. 

Figure \ref{fig:ddU} shows how the number of constraints in the definition of the uncertainty set $\cU$ influences the running time.

\begin{figure}[hbt!]
\centering
\includegraphics[width=0.8\textwidth]{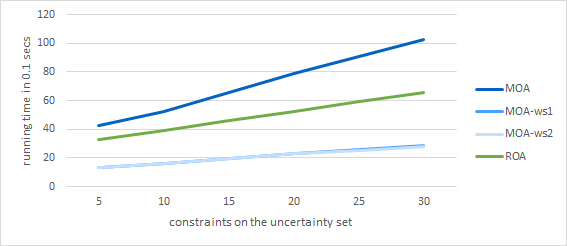}
\caption{Average running time of our four algorithms for 100 instances as a function of $m^\prime$ with $n=m=n^\prime=5$ fixed, $\cX$ and $\cU$ discrete} \label{fig:ddU}
\end{figure}

We observe the same pattern: The modified warm-start versions of MOA are by far the fastest algorithms; ROA is still faster than the baseline version of MOA. Clearly, the problem gets harder the more constraints are necessary to describe $\cU$. 
This leads us to conclude that the difficulty of the problem is rooted much more in the complexity of $\cU$ than in the one of $\cX$.

\paragraph{Discrete feasible set and polytopal uncertainty set}
Now let us turn to problems with a polytope as uncertainty set.
On those instances all of the algorithms we introduced can be used.
This includes the dualization approach (DA), which is the only algorithm that does not use optimization-pessimization but instead solves the scalarized problem $P(\cU,\lambda)$ for each weight $\lambda$ directly (via the means of dualization of the inner problem).

Figures~\ref{fig:dcX} and~\ref{fig:dcU} show the average running time of our algorithms on the same instances as in Figures~\ref{fig:ddX} and~\ref{fig:ddU} -- just with the integrality constraint for $\cU$ dropped.

\begin{figure}[hbt!]
\centering
\includegraphics[width=0.8\textwidth]{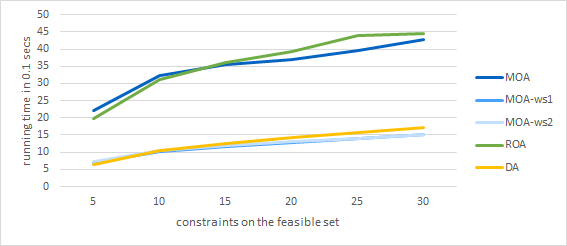}
\caption{Average running time of our five algorithms for 100 instances as a function of $n^\prime$ with $n=m=m^\prime=5$ fixed, $\cX$ discrete, $\cU$ polytope} \label{fig:dcX}
\end{figure}

\begin{figure}[hbt!]
\centering
\includegraphics[width=0.8\textwidth]{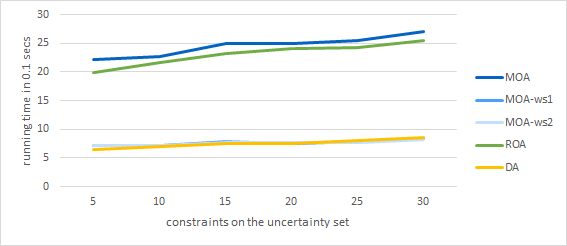}
\caption{Average running time of our five algorithms for 100 instances as a function of $m^\prime$ with $n=m=n^\prime=5$ fixed, $\cX$ discrete, $\cU$ polytope} \label{fig:dcU}
\end{figure}

Our experiments show that for such instances DA is effective, but not noticeably better than the modified versions of MOA.
The ranking of the other algorithms is essentially the same as before: The modified warm-start versions of MOA outperform ROA which is still faster than MOA's baseline version.
Dropping the integrality constraint reduced the overall running time of all algorithms by about factor two. This is while the number of extreme supported nondominated points stayed roughly the same.

The apparent ranking of the proposed algorithms raises the question of whether this applies only on average over a larger number of instances, or if it also applies to each individual instance.
For this we turn to Figure~\ref{fig:experimente-time-for-n-instances}.
In this figure we display the objective values for the 5 different algorithms on the first 10 of the tested 100 instances. 
Including all tested instances here does not change the discussed findings, but decreases visibility, which is why we included only the results of ten instances.

\begin{figure}[hbt!]
\centering
\includegraphics[width=0.8\textwidth]{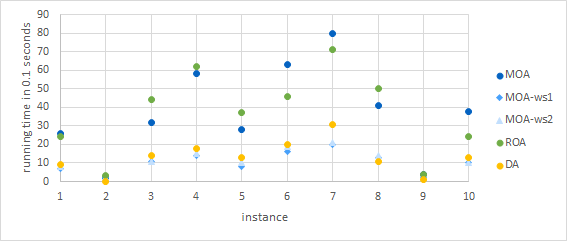}
\caption{Running time of our five algorithms for 10 instances with $n=n^\prime=m=5, m^\prime=30$, $\cX$ discrete, $\cU$ polytope} \label{fig:experimente-time-for-n-instances}
\end{figure}

Each of the ten columns in Figure~\ref{fig:experimente-time-for-n-instances} represents one instance (with $n=m=n^\prime=5, m^\prime=30$) on which we tested the algorithms.
We can see that for all instances either DA or the warm-start modifications of MOA perform best and either ROA or the baseline version of MOA perform worst.
The ranking of the algorithms is not the same for all instances.

To get a deeper understanding of this we turn to
Figure~\ref{fig:experimente-time-vs-scenarios}.

\begin{figure}[hbt!]
\centering
\includegraphics[width=0.8\textwidth]{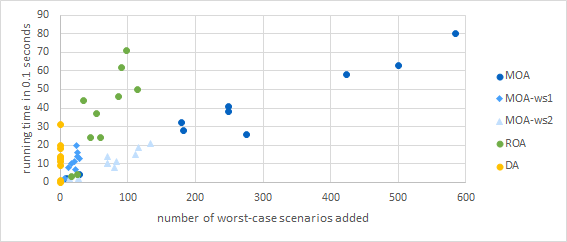}
\caption{Running time vs. number of worst-case scenarios added for 10 instances with $n=n^\prime=m=5, m^\prime=30$, $\cX$ discrete, $\cU$ polytope} \label{fig:experimente-time-vs-scenarios}
\end{figure}

For the ROA and all three versions of MOA it shows
the running time plotted against the number of times we add a worst-case scenario during the execution of the algorithms.
The strong correlation indicates that the number of pessimization steps decisively determines the overall time required.
The two algorithms where the uncertainty set $\cU^{(k)}$ grows monotonously, namely MOA-ws1 and ROA, have similarly high costs per added scenario. This can be explained by the fact that the resulting robust optimization problems are harder to solve due to the number of scenarios in $\cU^{(k)}$. 
Vice versa, MOA and MOA-ws2 both ``forget'' scenarios. Consequently, they need to (re)add more scenarios, but the optimization problems are simpler. For them the ratio between runtime and added scenario is lower.
This also explains why the warm-start modifications pay off: Apparently, the additional cost of starting with a larger scenario set $\cU^{(k)}$ is more than offset by less frequent need to execute of the pessimization step.

\paragraph{Evaluation for polytopal feasible sets}
Additionally, we tested the algorithms on instances with feasible sets $\cX$ that are polytopes.
In this case DA is faster. Apart from that, the observations do not deviate
significantly from the ones discussed in the previous paragraphs except that if $\cU$ is a polytope too, DA is faster than MOA-ws1 and MOA-ws2 as can be seen in Figure~\ref{fig:ccU}.

\begin{figure}[h!]
\centering
\includegraphics[width=0.8\textwidth]{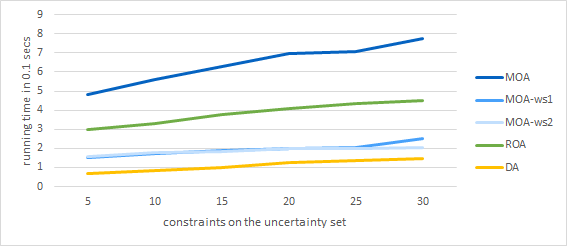}
\caption{Average running time of our five algorithms for 100 instances as a function of $m^\prime$ with $n=m=n^\prime=5$ fixed, $\cX$ and $\cU$ polytopes} \label{fig:ccU}
\end{figure}


\paragraph{The algorithms as approximation algorithms}
Lastly, we want to investigate how soon the algorithms provide a reasonable approximation of the Pareto front.
For this we turn to Algorithm~\ref{algo:ROA}, which in the $k$-th iteration determines (via dichotomic search) all extreme supported nondominated points of $BRO(\cU^{k})$ and then determines the worst-case outcomes of those points under $\cU$.
Figure~\ref{fig:experimente-bounds-ROA} shows for an instance with $n=m=n^\prime=5, m^\prime=30$ and $\cX$, $\cU$ both continuous, the lower and upper bound determined in the second and fourth iteration and the robust solutions determined in the final 7th iteration.
We can see that our method provides a good approximation to the Pareto front early on.
\begin{figure}[h!]
\centering
\includegraphics[width=0.8\textwidth]{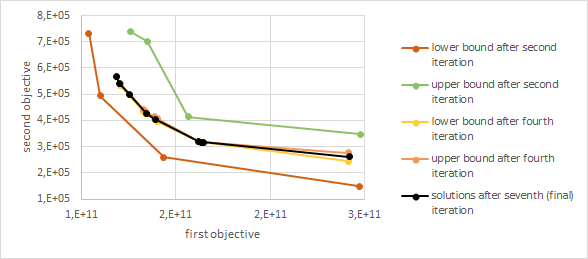}
\caption{Lower bound and upper bound determined in the second, fourth iteration and in the final (7th) iteration in an instance with $n=m=10, n^\prime=m^\prime=20$ and $\cX$, $\cU$ continuous} \label{fig:experimente-bounds-ROA}
\end{figure}

\section{Conclusions and further research} \label{sec:Conclusion}

In this paper, we have shown how biobjective mixed-integer linear optimization problems, where both objective functions are the maximum of a set of linear objective functions, can be solved.
While we framed this as a method for \emph{robust} biobjective optimization 
-- specifically to determine \emph{point-based minmax robust efficient} solutions for biobjective mixed-integer linear robust optimization problems --, 
our methods are not limited to such problems.
They can be applied to any biobjective optimization problem of the described structure.

Our solution method combines a well-known approach from biobjective optimization, namely dichotomic search, with approaches used in robust optimization, namely optimization-pessimization and reformulation.
In our numerical experiments, it has be shown that all our approaches are sensible for some problems.
We illustrate which approach is most suitable for which situation:
The robust optimizer's approach provides a good approximation of the set of extreme supported efficient solutions already early on; the warm-start modifications improve the multiobjective optimizer's approach such that it is fastest on instances where $\cU$ is discrete.
If $\cX$ and $\cU$ are polytopes, the dualization approach is the fastest.

Many avenues for further research exist that use the framework that we developed:
First, other and more advanced solution methods for multiobjective optimization can be used.
More specifically, dichotomic search can be replaced by any other enumeration method for extreme nondominated points (such as the one proposed in \cite{BoeklerMutzel2015}).
That way, a method similar to the one proposed in this paper can be used for problems with more than two objectives.
Similarly, solution methods for specific problems such as the multiobjective knapsack or the multiobjective TSP (see \cite{Vise1998TwophasesMA, EhrgottBook}) can be combined with optimization-pessimization to find robust solutions of these problems.

Second, extension to other robustness concepts for multiobjective optimization, such as set-based minmax robust efficiency, would be desirable. 
We plan to adapt the presented algorithms to the concept of regret robust efficiency (see \cite{GroetznerWerner22}).

\section*{Acknowledgments}

Fabian Chlumsky-Harttmann was supported by the DFG Research Grant ``Robust Multi-Objective Optimization: Analysis and Approaches''.

\printbibliography

\end{document}